\documentclass[pamq]{ipart}

\RequirePackage{hyperref}

\usepackage{mathrsfs}

\startlocaldefs
\newtheorem{Cor}[equation]{Corollary}
\newtheorem{Thm}[equation]{Theorem}
\newtheorem*{theorem*}{Theorem}
\newtheorem{Def}[equation]{Definition}
\newtheorem{def-prop}[equation]{Definition-Proposition}
\newtheorem{prop}[equation]{Proposition}
\newtheorem{prop-def}[equation]{Proposition-Definition}
\newtheorem{Ex}[equation]{Example}
\newtheorem{Rem}[equation]{Remark}

\numberwithin{equation}{section}

\newcommand{\abs}[1]{\lvert#1\rvert}

\newcommand{\degree}[1]{\abs{#1}}

\newcommand{\id}{\operatorname{id}}
\newcommand{\R}{\mathbb{R}}
\newcommand{\Z}{\mathbb{Z}}
\newcommand{\N}{\mathbb{N}}

\newcommand{\g}{\mathfrak{g}}
\newcommand{\h}{\mathfrak{h}}

\newcommand{\bd}{\begin{displaymath}}
	\newcommand{\ed}{\end{displaymath}}
\newcommand{\be}{\begin{equation}}
	\newcommand{\ee}{\end{equation}}

\newcommand{\Sh}{\mathrm{Sh}}

\newcommand{\K}{\mathbb{K}}

\newcommand{\bsection}[1]{\Gamma(#1)}

\newcommand{\dBott}{d^\bott_A}

\def\bott{\mathrm{Bott}}
\def\bp{\begin{proof}}

\def\g{\mathfrak{g}}

\def\m{\mathfrak{m}}

\def\v{\mathscr{A}}

\def\id{\operatorname{id}}
\def\Der{\operatorname{Der}}
\def\Hom{\operatorname{Hom}}

\newcommand{\Linfty}{ L_\infty  }

\newcommand{\piepartial}{\eth}

\newcommand{\inputvariable}{\cdot}

\newcommand{\pairing}[2]{\langle #1,#2\rangle}
\newcommand{\rank}{\mathrm{rank}}
\newcommand{\dA}{d_A}
\newcommand{\Linftythree}{L_{\leqslant 3}}%{L(3)}%

\newcommand{{\newboundary}}{\kappa}
\newcommand{\Derivations}{\mathrm{Der}}
\newcommand{\CinfMK}{C^\infty(M,\K)}

\newcommand{\projection}{\mathrm{pr}}
\newcommand{\projectionA}{\mathrm{pr}_A}

\newcommand{\projectionB}{\mathrm{pr}_B}

\newcommand{\deltap}{\delta}
\newcommand{\deltas}{s}

\newcommand{\adLb}{{\mathrm{ad}_b  }}

\newcommand{\maximalidealofartin}{\m_\v}

\endlocaldefs

\pubyear{202X}
\volume{X}
\issue{X}
\firstpage{1}
\lastpage{}
\begin{document}

\begin{frontmatter}

%%  "Title of the Paper"
\title[Internal symmetry of the $L_{\leqslant 3}$ algebra arising from a Lie pair]{Internal symmetry of the $L_{\leqslant 3}$ algebra \\ arising from a Lie pair}

\begin{aug}
%%  \author{\fnms{John} \snm{Smith}\thanksref{t2}\ead[label=e1]{smith@foo.com}\ead[label=e2,url]{www.foo.com}}
%%  \thankstext{t2}{The author is supported by ...}
%%  \address{line 1\\ line 2\\ \printead{e1}\\\printead{e2}}

\author{\fnms{Dadi} \snm{Ni}\ead[label=e1]{ndd18@mails.tsinghua.edu.cn}},
\address{Department of Mathematics\\ Tsinghua University\\Beijing\\ China\\
	\printead{e1}}
\author{\fnms{Jiahao} \snm{Cheng}\ead[label=e2]{jiahaocheng@nchu.edu.cn}},
\address{College of Mathematics and Information Science / Center for Mathematical Sciences\\ Nanchang Hangkong University\\Nanchang\\ China\\
	\printead{e2}}
\author{\fnms{Zhuo}
\snm{Chen}\ead[label=e3]{chenzhuo@mail.tsinghua.edu.cn}},
\address{Department of Mathematics\\ Tsinghua University\\Beijing\\China\\
\printead{e3}}
\and
\author{\fnms{Chen}
	\snm{He}\thanksref{t2}\ead[label=e4]{che@ncepu.edu.cn}}
\address{School of Mathematics and Physics\\ North China Electric Power University\\Beijing\\ China\\	
	\printead{e4}}

\thankstext{t2}{Corresponding author.}

{\small\center\textit{Dedicated to Prof. Victor Guillemin on the occasion of his 85th birthday}}
%\author{\fnms{???} \snm{???}\thanksref{t2}\ead[label=e2]{???}}
%\address{???\\\printead{e2}}
%\and
%\author{\fnms{???} \snm{???}%
%        \ead[label=e3]{???}%
%        \ead[label=u1,url]{???}}
%\address{???\\\printead{e3}\\\printead{u1}}
%\thankstext{t2}{The author is supported by ...}
\end{aug}
%%  History:
%\received{\sday{3} \smonth{1} \syear{2019}}

\begin{abstract}
	A Lie pair is an inclusion $A$ to $L$ of Lie algebroids over the same base manifold. 
	In an earlier work, the third author with Bandiera, Sti\'{e}non, and Xu 
	introduced a canonical $L_{\leqslant 3}$ algebra
	$\Gamma(\wedge^\bullet  A^\vee \otimes  L/A)$ whose unary bracket is  
	the Chevalley-Eilenberg differential arising from every Lie pair $(L,A)$.
	In this note, we prove that to such a Lie pair there is an associated
	Lie algebra action by $\Derivations(L)$ on the $L_{\leqslant 3}$ algebra
	$\Gamma(\wedge^\bullet  A^\vee \otimes  L/A)$.
	Here $\Derivations(L)$ is the space of derivations on the Lie algebroid $L$, or infinitesimal automorphisms of $L$.
	The said action gives rise to a larger scope of gauge equivalences of Maurer-Cartan elements in
	$\Gamma(\wedge^\bullet  A^\vee \otimes  L/A)$, and for this reason we
	elect to call the $\Derivations(L)$-action internal symmetry of
	$\Gamma(\wedge^\bullet  A^\vee \otimes  L/A)$.
\end{abstract}


%%  Upper case for every keyword

%\tableofcontents

\end{frontmatter}

\section{Introduction}

The present note is a continuation of the previous work \cite{BCSX}, in which   some homotopy-level structures from Lie pairs, known as $\Linfty$ algebroids,   shifted derived Poisson algebras and so forth were found. Now, we are more concerned with   intrinsic properties of these structures. The motivation here is detailed below.

Lie algebroids were introduced in the 1960s by Pradines \cite{Pradines} as a formalization of ideas
going back to the works of Lie and Cartan.  Recall that a Lie $\K$-algebroid
($\K$ is either $\mathbb{R}$ or $\mathbb{C}$) is a $\K$-vector bundle $L\to M$ over a smooth manifold $M$ such that $\Gamma(L)$ is a Lie-Rinehart $\K$-algebra \cite{Rinehart} over the commutative ring $C^\infty(M,\K)$. Namely, $\bsection{L}$ is equipped with a Lie bracket $[\cdot,\cdot]_L$, an anchor map    $\rho_L\colon L\to TM\otimes_{\R}\K$, and they are compatible (see Definition \ref{Def:LieKalgebroid}). In recent years, people tend to give the equivalent description of the Lie algebroid structure on $L$  from a dg manifold point of view \cite{Vaintrob1997}, namely, the Chevalley-Eilenberg differential operator $d_L:\bsection{\wedge^\bullet  L^\vee}\to \bsection{\wedge^{\bullet+1}  L^\vee}$ encodes $[\cdot,\cdot]_L$ and  $\rho_L$, and gives rise to a 
dg algebra $(\bsection{\wedge^\bullet  L^\vee},d_L)$.

Objects of Lie algebroids interpolate between tangent bundles, foliations
on the one hand and on the other
hand, Lie algebras and their actions on manifolds. As Lie algebroids combine the usual differential geometry and Lie
algebra theory under a common roof, it
suggests their relations with the realm of Mechanics. In fact, they form a category which is closely related to symplectic and Poisson manifolds, and there has been recently a lot of work and progress in the
geometric aspects of Hamiltonian  and Lagrangian  mechanics on Lie algebroids, see  e.g.  \cite{Grabowska,deLeon,Martinez,Weinstein}. 

We say that $(L, A)$ is a pair of Lie algebroids or \textit{Lie pair} for short if $A$ is a subalgebroid of another Lie algebroid $L$ over the same base manifold. The notion of  {Lie pair} is a natural framework encompassing a range of diverse geometric 
contexts including complex manifolds, foliations, matched pairs, and manifolds
endowed with an action of a Lie algebra, etc. 
In the last decade much research on Lie pairs has been done  following different  strategies and the underlying mathematical structures: Atiyah classes arising from Lie pairs have been studied, using a variety of methods, see e.g. \cite{CSX16,CXX,Batakidis};    
It is  shown that geometric objects including  Kapranov dg and Fedosov dg manifolds \cite{LSXadvance2021,SX20}, 
algebraic objects such as Hopf algebras \cite{CSX14,CCN},   
Leibniz$_\infty$ and $\Linfty$ algebras   can be derived from Lie pairs \cite{LSX12,BCSX,CLX}; 
Also, in the context of Lie pairs, considerable attentions had been paid to   Poincar\'{e}-Birkhoff-Witt isomorphisms \cite{CCT,Calaque},   Kontsevich-Duflo isomorphisms \cite{LSX,CXXarxiv2021}, and Rozansky-Witten-type invariants \cite{VX}, etc.

We remind the reader briefly of
the notion of Lie algebroid representation of $L$ on a vector bundle $V$ (over the same base manifold), also known as an $L$-action, or an $L$-module structure on $V$ ---  This means a $\K$-bilinear map $\nabla:~\bsection{L}\times \bsection{V}\to \bsection{V} $ which is $\CinfMK$-linear in its first argument and satisfying
\begin{enumerate}
	\item $\nabla_{l}(fv)=f\nabla_l v+ \rho_L(l)(f) v$, for all $l\in \bsection{L}$, $f\in \CinfMK$ and $v\in\bsection{V}$;
	\item $\nabla_{[l_1,l_2]_L}=\nabla_{l_1}\circ \nabla_{l_2}-\nabla_{l_2}\circ \nabla_{l_1}$, for all $l_1,l_2\in \bsection{L}$.	
\end{enumerate}  
If we turn to the language of dg algebras, such an action $\nabla$ is equivalent to a square zero derivation 
$d^\nabla_L$: $\bsection{\wedge^\bullet  L^\vee\otimes V}\to \bsection{\wedge^{\bullet+1}   L^\vee\otimes V}$ which is also called the Chevalley-Eilenberg differential  \cite{Vaintrob1997}. 
The pair $\big(\bsection{\wedge^\bullet  L^\vee\otimes V},d^\nabla_L\big)$ 
now becomes a dg module over the dg algebra
$(\bsection{\wedge^\bullet  L^\vee},d_L)$, and is known as the 
Chevalley-Eilenberg complex of $L$ with coefficients in $V$.

Notably, given a Lie pair $(L, A)$ over a base manifold $M$,
the quotient bundle $L/A$, which we usually denote by $B$,
naturally admits a Lie algebroid $A$-action known
as the Bott connection \cite{Bott1970,Bott}.
In specific, this action  which we denote by $\nabla$ is defined by  
\[\Gamma(A)\times\Gamma(B)\longrightarrow \Gamma(B): \quad (a,b)\longmapsto \nabla_a b:={\rm pr}_B[a,\tilde{b}]_L\]
where $a\in \bsection{A}$, $b\in \bsection{B}$, ${\rm pr}_B: L\to B$ stands for the canonical quotient map, and $\tilde{b}\in \bsection{L}$ satisfies ${\rm pr}_B(\tilde{b})=b$. 
The  Bott connection originates in
foliation theory: When we especially take $L = TM$  with $A$ the tangent bundle to
a foliation on $M$, and $B=L/A$ the normal bundle to the foliation, the action $\nabla$ becomes the canonical flat connection on $B$ along $A$ considered by Bott \cite{Bott1970}.

From the Lie algebroid $A$ and its action on $B$, we have the standard dg algebra 
$\Omega^\bullet _A:=\Gamma(\wedge^\bullet  A^\vee)$ and the dg module 
$\Omega_A^\bullet(B):= \Gamma(\wedge^\bullet  A^\vee \otimes L/A)$
(both equipped with Chevalley-Eilenberg  differentials). To interpret 
$\Omega_A^\bullet(B)$, consider  $\mathfrak{X}_A$ and $\mathfrak{X}_L$, the 
differentiable stacks determined by the local Lie groupoids
integrating the Lie algebroids $A$ and $L$, respectively.
The dg module $\Omega_A^\bullet(B)$ can be regarded as the space of formal  
vector fields tangent to the fibers of the differentiable 
stack fibration $\mathfrak{X}_A\to \mathfrak{X}_L$. 
We would like to mention another point of view towards foliations
on manifolds, which are particular instances of Lie pairs,
due to Vitagliano \cite{V2009,V2014,V2015}  ---  
Let $C\subset TM$ be an involutive  distribution on a finite dimensional smooth manifold $M$. This gives    an object of a \textit{diffiety} to which there is associated a rich cohomological calculus, also known as secondary calculus \cite{Vinogradov1984,Vinogradov1998,Vinogradov2001}.
For instance,
secondary functions give characteristic (de Rham) cohomology of $C$, i.e. $H_{\mathrm{dR}}\big(\Gamma(\wedge^\bullet  C^\vee )\big)$, and secondary vector fields give characteristic
cohomology with local coefficients in $C$-normal vector fields, i.e. 
$H_{\mathrm{dR}}\big(\Gamma\big(\wedge^\bullet  C^\vee \otimes TM/C\big)\big)$.  
Although the settings \textit{op. cit.} are foliations on smooth 
manifolds, we can adapt them to Lie pairs easily: given a general Lie pair 
$(L,A)$, one interprets the Chevalley-Eilenberg cohomology 
$H_{\mathrm{CE}}(A):=H\big(\Gamma(\wedge^\bullet A^\vee)\big)=H(\Omega_A^\bullet)$ as the secondary functions and  
$H_{\mathrm{CE}}(A;B):=H\big(\Gamma(\wedge^\bullet A^\vee \otimes B)\big)=H\big(\Omega_A^\bullet(B)\big)$
as the secondary vector fields stemming from a stack fibration 	
$\mathfrak{X}_A\to \mathfrak{X}_L$.

To the aforementioned space $\Omega_A^\bullet(B)$, there are associated two \emph{different} $\Linfty$ structures, one by \cite{LSX12} and the other by \cite{BCSX}. The latter one is relatively easier because it is indeed an $\Linftythree$ algebra (i.e., its structure  maps $[\cdots]_n$ are trivial for $n\geqslant 4$, see  Section \ref{Sec:LinftyLiePair}).
We wish to find more intrinsic properties of the both $\Linfty$ structures  but in this note we only handle the latter one and give the following result   --- the Lie algebra $\Derivations(L)$ of derivations of the Lie algebroid $L$ has an action on the said $\Linftythree$ algebra $\Omega_A^\bullet(B)$, see  Theorem \ref{MainTheorem}.

Why the $\Linftythree$-algebra $\Omega_A^\bullet(B)$, the central object of this study, is interesting? 
In fact, it is shown in \cite{BCSX} that $\Omega_A^\bullet(B)$ is a `resolution' of $H_{\mathrm{CE}}(A;B)=H\big(\Omega_A^\bullet(B)\big)$,
which admits a canonical graded Lie algebra structure.
As a consequence of our Theorem \ref{MainTheorem},
we find an \textit{action by derivation} of a Lie subalgebra of 
$\Derivations(L)$ on $H_{\mathrm{CE}}(A;B)$,
see Corollary \ref{Cor:actiononH}. Moreover, $\Omega_A^\bullet(B)$ is closely related with deformations of the Lie pair $(L,A)$. To be more precise,  a Maurer-Cartan element in  $\Omega_A^\bullet(B)$, i.e., a certain element  $\xi\in \Omega_A^1(B)\otimes \maximalidealofartin  $ (where $\maximalidealofartin $ is the maximal ideal of a local Artinian  $\K$-algebra $\v$; see Definition
\ref{Def:MC}), gives rise to a deformed Lie pair $(L,A_{\xi})$. It turns out that the aforementioned action by $\Derivations{(L)}$ on $\Omega^\bullet _A(B)$ determines a family of gauge equivalences of  such deformations. Limited by the length of this paper, we do not expand this relatively complicated topic. Detailed investigation and conclusions will be presented in our next work.

To prove Theorem \ref{MainTheorem}, we first need to specify the notion of a Lie algebra action on an $\Linfty$ algebra. We will borrow many known facts from the literature and give a formal definition and several  equivalent characterizations which should be of independent interest, see  Section \ref{Sec:LiealgebraactiononLinfty}.  We would like to point out that, in \cite{MehtaZambon},  Mehta and Zambon have introduced the concept of   $L_\infty$ algebra action on a graded manifold, which is more general than our definition of Lie algebra actions and could be a point for further studies  related to the present work.  

Further more, we propose to consider gauge actions on Maurer-Cartan elements in an $\Linfty$ algebra if it adopts a Lie algebra action. The idea follows from the original definition of gauge actions by Getzler \cite{Getzler}. When applying to the particular instance of $\Derivations(L)$-action on $\Omega_A^\bullet (B)$, we obtain a new type of gauge equivalence of Maurer-Cartan elements in  $\Omega_A^\bullet (B)$, which is compatible with the standard gauge equivalence in the sense of Getzler, see  Theorem \ref{Thm:adb=baction}. For this reason, we elect to call the action  by   $\Derivations(L)$    \textit{internal symmetry} of the $\Linftythree$  algebra $\Omega_A^\bullet (B)$.

\smallskip

\textbf{Terminology and notations.}

\textit{Field $\K$.} We use the symbol $\K$ to denote the field of either real or complex numbers.

\textit{Gradings.} 
Unless specified otherwise, all vector spaces, algebras, modules, etc., are 
$\Z$-graded objects over the field $\K$.
For a graded object $V=\oplus_{i\in \Z} V^{i}$, the degree of
a homogeneous element $v\in V^{i}$ is denoted by $\degree{v}=i$.
We write $V[k]$ as the degree $k$ shift of $V$ by the rule 
$(V[k])^i=V^{k+i}$, so the element $v[k]\in V[k]$ shifted from
$v\in V$ has degree $\degree{ {v[k]}} =\degree{v}-k$.
The elements $v[1]\in V[1]$ will also be denoted by $\tilde{v}$.

\textit{Tensor products.}
Given a graded vector space $V$, we denote its tensor (co)algebra and 
exterior (co)algebra over $\K$ respectively by 
$T(V):=\bigoplus_{n\geqslant  0}V^{\otimes n}$ and 
$\Lambda(V):=\bigoplus_{n\geqslant  0}\Lambda^n(V)$.
The symmetric (co)algebra and the reduced symmetric (co)algebra of $V$ are respectively denoted
by $S(V):=\bigoplus_{n\geqslant  0}S^n(V)$ and $\overline{S(V)}:=\bigoplus_{n\geqslant 1}S^n(V)$, whose multiplication symbols are written as $\odot$.

\textit{Shuffles.}  Let $\mathfrak{S}_n$ denote the permutation group of the set $\{1, 2, \cdots, n\}$. 
A $(p, \, q)$-shuffle is a permutation
$\sigma\in \mathfrak{S}_{p+q}$ 
such that
$\sigma(1)<\cdots < \sigma(p)$
and
$\sigma(p+1)<\cdots < \sigma(p+q)$.
The symbol $\mathrm{Sh}(p, \, q)$ denotes the set of $(p, \, q)$-shuffles.
Similarly, 
$\Sh(i,j,k)$ denotes the set of such $(i,j,k)$-shuffles, i.e., those  $\sigma \in \mathfrak{S}_{i+j+k}$  satisfying  $\sigma(1)<\cdots<\sigma(i)$,
$\sigma(i+1)<\cdots<\sigma(i+j)$, and $\sigma(i+j+1)<\cdots<\sigma(i+j+k)$.
For any shuffle $\sigma$, its sign is denoted by ${\rm sgn}(\sigma)$.

\textit{Koszul signs.}  
For homogeneous elements $v_1,\cdots,v_n\in V$ and 
$\sigma\in\mathfrak{S}_n$, the Koszul signs 
$\epsilon(\sigma;v_1,\cdots,v_n)$ and $\chi(\sigma;v_1,\cdots,v_n)$ are 
defined respectively by the equations
\begin{align*}
	v_1\odot\cdots\odot v_n&=\epsilon(\sigma;v_1,\cdots,v_n)v_{\sigma(1)}\odot\cdots\odot v_{\sigma(n)} \quad(\in S(V)),\\\mbox{ and }~ 
	v_1\wedge\cdots\wedge v_n&=\chi(\sigma;v_1,\cdots,v_n)v_{\sigma(1)}\wedge\cdots\wedge v_{\sigma(n)}\quad(\in \Lambda(V)).
\end{align*}
The two signs are related by
$\chi(\sigma;v_1,\cdots,v_n)={\rm sgn}(\sigma)\epsilon(\sigma;v_1,\cdots,v_n)$.

\textit{The d\'{e}calage isomorphism.} Given a graded vector space $V$, there are natural isomorphisms for any $n\in \N$:
\begin{equation*}
	{\rm dp}_n:~{S^n}(V[1])\to (\Lambda^nV)[n],\quad \tilde{v}_1\odot\cdots \odot \tilde{v}_n\mapsto (-1)^{\sum_{i=1}^n(n-i)\degree{v_i}}(v_1\wedge\cdots\wedge v_n)[n].
\end{equation*}

\textit{Lie algebroids.}
In this paper `Lie algebroid' always means
`Lie $\K$-algebroid'.

\textit{Abbreviations.}  The word `dg' stands for `differential graded'. Likewise, `dgla' stands for `differential graded Lie algebra'.
\subsection*{Acknowledgement}
We would like to thank Chuangqiang Hu and Maosong Xiang for useful comments. We gratefully acknowledge the critical review by the anonymous referee on an earlier version of the manuscript.
The research is supported by 
NSFC grant  12071241, Research Fund of Nanchang Hangkong University (EA202107232), and	the Fundamental Research Funds for the Central Universities (2020MS040).

%We gratefully acknowledge the critical reviews by the two anonymous reviewers on earlier versions of the manuscript.

\section{$L_\infty$ algebras}
\subsection{Definitions of $L_\infty$ algebras} 
We start with the definitions of various terms related to homotopy Lie algebras.
\begin{Def}\label{Def:Linfty-algebra}
	An \textbf{$L_\infty$ algebra}  is a graded vector space $\g$ equipped with a collection of skew-symmetric multilinear maps $[\cdots]_k: \Lambda^k \g\to \g $ of degree $(2-k)$, for all $k\geqslant  $1$ $, such that the higher Jacobi rules
	\smallskip
	\begin{equation}\label{Jacobi identities}
		\sum\limits_{{\begin{subarray}{c}
					i=1,\cdots,n \\
					\sigma \in \Sh(i,n-i)\end{subarray}}}  (-1)^i\chi(\sigma;x_1,\cdots,x_n) [[x_{\sigma(1)},\cdots,x_{\sigma(i)}]_i,x_{\sigma(i+1)},\cdots, x_{\sigma(n)}]_{n-i+1}=0,
	\end{equation}
	hold for all homogeneous elements $x_1,\cdots,x_n \in V$ and $n\geqslant  1$.
\end{Def}

\textbf{Notation:}  It is common to denote the unary bracket $[\inputvariable]_1$ by $d$.

The first three Jacobi rules are listed below:
\begin{itemize} \label{eq:3Jacobi}
	\item[$\bullet$] ($n=1$)	 $d^2=0,$
	\item[$\bullet$] ($n=2$) $d[x_1,x_2]_2=[d x_1,x_2]_2+(-1)^{\degree{x_1}}[x_1,d x_2]_2,$
	\item[$\bullet$] ($n=3$)
	\begin{align*}
		& [[x_1,x_2]_2,x_3]_2+(-1)^{1+\degree{x_2}\cdot\degree{x_3}}[[x_1,x_3]_2,x_2]_2\\
		&\qquad\qquad+(-1)^{\degree{x_1}(\degree{x_2}+\degree{x_3})}[[x_2,x_3]_2,x_1]_2\\
		=\,& d [x_1,x_2,x_3]_3+[d x_1,x_2,x_3]_3+(-1)^{1+\degree{x_1}\cdot\degree{x_2}}[d x_2, x_1,x_3]_3\\
		&\qquad\qquad+(-1)^{(\degree{x_1}+\degree{x_2})\degree{x_3}}[d x_3,x_1,x_2]_3.
	\end{align*}
\end{itemize}

In particular, an $L_\infty$ algebra  with $[\cdots]_k=0$ for $k\geqslant  3$ amounts to
a differential graded Lie algebra (dgla  for short), i.e., a triple $(V,[\inputvariable,\inputvariable],d)$,
where $(V=\bigoplus_{i\in\Z} V^i,d)$ is a differential graded vector space and $[\cdot,\cdot]\colon V\times V \to V$ is a graded Lie bracket, satisfying the above $n=2$ constraint.

\begin{Def}
	If an $L_\infty$ algebra has vanishing brackets $[\cdots]_k=0$ for $k\geqslant  4$, i.e.,
	 only the brackets $d$, $[\cdot,\cdot]_2$, and 
	$[\cdot,\cdot,\cdot]_3$ are possibly nontrivial, then it is called an
	\textbf{$L_{\leqslant 3}$ algebra}.
\end{Def}

\begin{Rem}
	In this paper, we follow the sign convention of Getzler \cite{Getzler} for the definition of $L_\infty$ algebras. The original definition of $L_\infty$ structure on $\g$ introduced by Lada and Markl \cite{LM95}
	means a family of brackets 
	$[\cdots]'_k: \Lambda^k(\g)\to \g $ of degree $(2-k)$ subject to the 
	higher Jacobi rules:
	\smallskip 
	\begin{equation}
	\begin{aligned}
		\label{Eqt:nJacabi-2}
		&\sum\limits_{{\begin{subarray}{c}
					i=1,\cdots,n \\
					\sigma \in \Sh(i,n-i)\end{subarray}}}  (-1)^{i(n-i)}  \chi(\sigma;x_1,\cdots,x_n)\\
		&\qquad\qquad\qquad[[x_{\sigma(1)},\cdots,x_{\sigma(i)}]'_i,x_{\sigma(i+1)},\cdots,
		x_{\sigma(n)}]'_{n-i+1}
		=0.	
	\end{aligned}
\end{equation}
	In fact, if we define
	$$[x_1,\cdots,x_k]_k:=(-1)^{\frac{k(k+1)}{2}} [x_1,\cdots,x_k]'_k,$$
	then the identity \eqref{Eqt:nJacabi-2} becomes
	\eqref{Jacobi identities}. 
\end{Rem}

\begin{Rem}\label{rmk:H(g)}
	For an $L_\infty$-algebra $\g$ whose unary bracket is $d$,
	since $d^2=0$, we have a cochain complex $(\g,d)$.
	The corresponding cohomology is denoted by $H(\g)$.
	Moreover, $H(\g)$ is equipped with a graded Lie algebra structure
	whose Lie bracket is induced from the binary bracket of $\g$. Note this $H(\g)$ is \emph{not}  $H_{\mathrm{CE}}(\g)$ which usually refers to the Chevalley-Eilenberg cohomology of the $L_\infty$ algebra $\g$.
\end{Rem}

\begin{Def}\label{Def:morphismLinfty}
	Let $(\g_1,[\cdots]^1_n)$ and $(\g_2,[\cdots]^2_n)$ be two $L_\infty$ algebras. 
	An \textbf{${L_\infty}$ morphism}   $F:~\g_1\to\g_2$ is a collection of multilinear maps $$F_n: \Lambda^n \g_1\to \g_2 $$
	(of degree $1-n$) 	satisfying the following equations:
	$$\begin{aligned}
		&\sum\limits_{{\begin{subarray}{c}
					p+q=n+1 \\
					\sigma \in \Sh(p,n-p)\end{subarray}}}\pm
		F_q\big([x_{\sigma(1)},\cdots,x_{\sigma(p)}]^1_p,
		x_{\sigma(p+1)},\cdots,x_{\sigma(n)}\big)\\
		&\quad=\sum\limits_{{\begin{subarray}{c}
					1\leqslant k \leqslant n \\
					i_1+\cdots+i_k=n\\
					\tau \in \Sh(i_1,\cdots,i_k)\end{subarray}}} \pm
		\big[F_{i_1}(x_{\tau(1)},\cdots,x_{\tau(i_1)}),\cdots,
		F_{i_k}(x_{\tau(i_1+\cdots+i_{k-1}+1)},\cdots,x_{\tau(n)})
		\big]^2_k
	\end{aligned}
	$$ for every $n\geqslant 1$ and $x_1,\cdots,x_n\in \g_1$. 
	Here we denote by $(\pm)$ the appropriate Koszul signs
	(which can be worked out explicitly).
\end{Def}

There is another version of definition for an $L_\infty$ algebra, known as $L_\infty[1]$ algebras, see \cite{Voronov,Schatz}.
\begin{Def}\label{Defn:Linftyone}
	An \textbf{$L_\infty[1]$ algebra}  is a graded vector space $W$ equipped with a collection of symmetric multilinear maps $\{\cdots\}_k: S^k(W)\to W$, all being degree $1$  for $k\geqslant  $1$ $, such that the higher Jacobi rules
	\smallskip
	\begin{equation}\label{Jacobi identities1}
	\sum\limits_{{\begin{subarray}{c}
				i=1,\cdots,n \\
				\sigma \in \Sh(i,n-i)\end{subarray}}} \epsilon(\sigma;w_1,\cdots,w_n) 
		\big\{\{w_{\sigma(1)},\cdots,w_{\sigma(i)}\}_i,w_{\sigma(i+1)},\cdots, w_{\sigma(n)}\big\}_{n-i+1}=0
	\end{equation}
	hold for all homogeneous elements $w_1,\cdots,w_n \in W$ and $n\geqslant  1$.
	%	We usually speak of the $L_\infty[1]$ algebra  $W$, without specifying the operation $\{\cdots\}_k $.
\end{Def}

A basic fact we need is the bijection between $L_\infty$ algebra structures on a graded vector space $\g$ and $L_\infty[1]$ algebra structures on $\g[1]$ --- Suppose that $(\g,[\cdots]_n)$ is an $L_\infty$ algebra. Following the d\'{e}calage isomorphism, we define the following $n$-brackets:
\begin{equation*}
	{S^n}(\g[1])\to \g[1],\quad
	\{\tilde{x}_1,\cdots,\tilde{x}_n\}_n:=(-1)^{\frac{n(n+1)}{2}+\sum_{i=1}^n(n-i)\degree{x_i}}([x_1,\cdots,x_n]_n)[1],
\end{equation*} where $x_1,\cdots,x_n\in \g$. 
Then one can verify that $(\g[1],\{\cdots \}_n)$ is an $L_\infty[1]$ algebra. 

Another relation we need is about the Koszul signs:
\begin{equation*}\label{Eqt:sign}
	(-1)^{\sum_{i=1}^n(n-i)\degree{x_{\sigma(i)}}}\epsilon(\sigma;\tilde{x}_1,\cdots,\tilde{x}_n)=(-1)^{\sum_{i=1}^n(n-i)\degree{x_i}}\chi(\sigma;x_1,\cdots,x_n),
\end{equation*}
where $\sigma\in \mathfrak{S}_n$,  $\epsilon(\sigma;\tilde{x}_1,\cdots,\tilde{x}_n)$ and $\chi(\sigma;x_1,\cdots,x_n)$ are defined in $S^n(\g[1])$ and $\Lambda^n(\g)$ respectively.

\subsection{$L_\infty$ structure in terms of dg coalgebras}

The notions of $L_\infty$ and $L_\infty[1]$ algebras can be wrapped up in  the language of dg coalgebras. Here is a quick review of this fact. One could also consult \cite{Manetti} and \cite{Fukaya}.

Let $C$ be a graded coalgebra with a comultiplication
$\Delta:~C\to C\otimes C$. A degree $m$ \textbf{coderivation} of 
$(C,\Delta)$ is a degree $m$ graded linear map
$D:~C\to C$ satisfying the coLeibniz rule
$\Delta D=(D\otimes\id_C+\id_C\otimes D)\Delta$. 
A \textbf{codifferential} is a degree $1$ coderivation $D$ with $D^2=0$.

\textbf{Notation:}  The vector space of all coderivations on $C$ is denoted by
${\rm Coder}(C):=\bigoplus_{m\in\Z}{\rm Coder}^m(C)$
which is naturally graded by the degrees.
%We say that an element $D\in{\rm Coder}^n(C)$ has degree $n$
%and we write $\degree{D}=n$.
With respect to the composition $\circ$, the space ${\rm Coder}(C)$
has the graded Lie bracket 
$[F,G]:=F\circ G-(-1)^{\degree{F}\degree{G}}G\circ F$  for homogeneous elements $F,G\in{\rm Coder}(C)$.
If $D:~C\to C$ is a codifferential, then $[\cdot,\cdot]$ and $d=[D,\cdot]$
give a dgla  structure on ${\rm Coder}(C)$.

We will mainly work with two coalgebras associated with a given 
graded vector space $V$: the symmetric coalgebra  
$S(V)=\bigoplus_{n\geqslant 0}S^n(V)$ and the reduced one 
$\overline{S(V)}=\bigoplus_{n\geqslant 1}S^n(V)$. 
The comultiplication on $S(V)$ is given as follows:
$$
\begin{aligned}
&\mathfrak{b}(v_1\odot\cdots\odot v_n)\\
&\quad= \sum\limits_{{\begin{subarray}{c}
			r=0,\cdots,n \\
			\sigma \in \Sh(r,n-r)\end{subarray}}}
		\epsilon(\sigma;v_1,\cdots,v_n)(v_{\sigma(1)}\odot\cdots\odot v_{\sigma(r)})\otimes(v_{\sigma(r+1)}\odot\cdots\odot v_{\sigma(n)}),
		\end{aligned}$$
where $v_1,\cdots,v_n\in V$ are homogeneous elements.
In particular, for a homogeneous element $v\in V$, we have 
$\mathfrak{b}(v)=1\otimes v+v\otimes 1$.
For $1\in S^0(V)=\K$, we have $\mathfrak{b}(1)=1\otimes 1$.
The counit $\varepsilon:~S(V)\to \K=S^0(V)$ is the natural projection.

Similarly, the comultiplication on the reduced symmetric coalgebra  
$\overline{S(V)}$ is given as follows:
$$
\begin{aligned}
&\mathfrak{l}(v_1\odot\cdots\odot v_n)\\
&\quad=\sum\limits_{{\begin{subarray}{c}
			r=1,\cdots,n-1 \\
			\sigma \in \Sh(r,n-r)\end{subarray}}} \epsilon(\sigma;v_1,\cdots,v_n)(v_{\sigma(1)}\odot\cdots\odot v_{\sigma(r)})\otimes(v_{\sigma(r+1)}\odot\cdots\odot v_{\sigma(n)}),
\end{aligned}$$
where $v_1,\cdots,v_n\in V$ are homogeneous elements.
In particular, for a homogeneous element $v\in V$, we have 
$\mathfrak{l}(v)=0$. Let $p:~S(V)\to \overline{S(V)}$ be the projection
with kernel $\K=S^0(V)$. It is easy to check that 
$\mathfrak{l}p=(p\otimes p)\mathfrak{b}$, 
i.e., $p$ is a morphism of graded coalgebras.

%\label{Lem:Fukaya}
Given a graded vector space $V$, every coderivation $D:~S(V)\to S(V)$ of degree $m$ is completely determined by its corestriction 
${\rm pr}_1\circ D=(D_0,D_1,D_2,$ $\cdots)$, where ${\rm pr}_1:~S(V)\to V$ denotes the canonical projection, and $D_k$ is the composition $S^k(V)\xrightarrow{D} S(V) \xrightarrow{{\rm pr}_1} V$, which is a graded vector space morphism of degree $m$. In other words,   there exists an isomorphism of graded vector spaces
\begin{equation}\label{ISO:Fukaya}
	{\rm Coder} (S(V))\cong\Hom (S(V),V),
\end{equation} 
whose inverse is given by the formula
\begin{align*}
	&D(v_1\odot \cdots\odot v_n)\\
	&\quad=D_0(1)\odot v_1\odot \cdots\odot v_n\\
	&\qquad+\sum\limits_{{\begin{subarray}{c}
				k=1,\cdots,n \\
				\sigma \in \Sh(k,n-k)\end{subarray}}}
	\epsilon(\sigma;v_1,\cdots,v_n)\\
	&\qquad\qquad\qquad\qquad D_k(v_{\sigma(1)}\odot\cdots\odot v_{\sigma(k)})\odot v_{\sigma(k+1)}\odot\cdots\odot v_{\sigma(n)},
\end{align*}
where $v_1,\cdots,v_n\in V$ are homogeneous elements.
Following this proposition, we use the notation 
$D=(D_0,D_1,D_2,\cdots)$ to denote a coderivation of $S(V)$. 

Similarly, for the reduced symmetric coalgebra $\overline{S(V)}$, 
the map $R\mapsto {\rm pr}_1\circ R=(R_1,R_2,\cdots)$
with $R_k:= {\rm pr}_1 \circ R|_{S^k(V)}$  
gives an isomorphism of graded vector spaces
\begin{equation}
	\label{ISO:reduced}{\rm Coder} (\overline{S(V)})\cong
	\Hom (\overline{S(V)},V),
\end{equation} 
whose inverse is given by the formula
$$
\begin{aligned}
&R(v_1\odot \cdots\odot v_n)\\
&\quad=\sum\limits_{{\begin{subarray}{c}
			k=1,\cdots,n \\
			\sigma \in \Sh(k,n-k)\end{subarray}}}
		\epsilon(\sigma;v_1,\cdots,v_n)R_k(v_{\sigma(1)}\odot\cdots\odot v_{\sigma(k)})\odot v_{\sigma(k+1)}\odot\cdots\odot v_{\sigma(n)},
		\end{aligned}$$
where $v_1,\cdots,v_n\in V$ are homogeneous elements.

By these two identifications, a coderivation 
$R=(R_1,R_2,\cdots)$ on $\overline{S(V)}$ is equivalent to a coderivation 
$R=(R_0,R_1,R_2,\cdots)$ on $S(V)$ with $R_0=0$.
On the other hand, given a coderivation $D=(D_0,D_1,D_2,\cdots)$ on $S(V)$,
we have a truncated coderivation $\overline{D}:=(D_1,D_2,\cdots)$ on
$\overline{S(V)}$.  In particular, if  $Q\in{\rm Coder}^1
(\overline{S(V)})$ is a codifferential, then $Q$ induces dgla  structures on both ${\rm Coder}(\overline{S(V)})$ and ${\rm Coder}(S(V))$, where the former is a sub-dgla  of the latter.

For a degree $i$ coderivation $R\in {\rm Coder}^i(\overline{S(V)})$
and a degree $j$ vector $v\in V^j$, we can define a degree $(i+j)$ 
coderivation $v\lrcorner R \in {\rm Coder}^{i+j}(\overline{S(V)})$ 
with components given by
\begin{equation*}
	\label{Eqt:contraction}
	(v\lrcorner R)_n(v_1\odot\cdots\odot v_n):=(-1)^{ij}R_{n+1}(v\odot v_1\odot\cdots\odot v_n),  
\end{equation*}
$ \forall n\geqslant 1, v_1\odot \cdots\odot v_n \in S^n(V)$.
One can check that 
%\begin{equation}
\begin{equation*}
	v\lrcorner R=-\overline{[v^\#,R]},  \quad   (-1)^{ij}R(v)=(-1)^{ij}R_1(v)=-{[v^\#,R]_0}(1)=-{[v^\#,R]}(1),
\end{equation*}
or equivalently,
\begin{equation}\label{eq:lrcorner}
	 -{[v^\#,R]}= \big((-1)^{ij}R(v)\big)^\#+v\lrcorner R , 
\end{equation}
%\end{equation}
where $v^\#=\big((v^\#)_0=v,0,0,\ldots\big)\in{\rm Coder}(S(V))$.

Finally, we need the fact that \textit{there is a bijection between $L_\infty$ 
	algebra structures on a graded vector space $\g$ and codifferentials on 
	the reduced symmetric coalgebra $\overline{S(\g[1])}$.} In fact, it suffices 
to consider $L_\infty[1]$ algebra structures on $\g[1]$. 
Suppose that a sequence of maps $\{\cdots\}_k: S^k(\g[1])\to \g[1]$,
all being degree $1$ for $k\geqslant 1$, define an $L_\infty[1]$ structure  
on $\g[1]$ (see Definition \ref{Defn:Linftyone}).
Set $Q_k=\{\cdots\}_k$ and consider $Q=(Q_1,Q_2,\cdots)\in  {\rm Coder}^1(\overline{S(\g[1])})$,
then the higher Jacobi identities \eqref{Jacobi identities1}
are equivalent to the identity $Q^2=0$, i.e., $Q$ is a codifferential 
on $\overline{S(\g[1])}$.

Throughout the note, we shall assume the correspondence between
an $L_\infty$ algebra $\g$ and its associated dg coalgebra  
$(\overline{S(\g[1])},Q)$ by default. Of course, the codifferential $Q$ also gives rise to a dg coalgebra $( {S(\g[1])},Q)$. 
In this context,  a morphism of $L_\infty$ algebras $\g_1 \to \g_2 $ translates to  a morphism of dg coalgebras $(\overline{S(\g_1 [1])},Q_1)\to (\overline{S(\g_2 [1])},Q_2)$, and $( {S(\g_1 [1])},Q_1)\to ( {S(\g_2 [1])},Q_2)$ as well.

\subsection{The $\Linftythree$ algebra  arising from a Lie pair}	\label{Sec:LinftyLiePair}
\begin{Def}\label{Def:LieKalgebroid}
	Let $M$ be a smooth manifold. A \textbf{Lie algebroid} over $M$ consists of a $\K$-vector bundle $E\to M$, a vector bundle map $\rho_E\colon E\to TM\otimes_{\R}\K$, called anchor, and a Lie algebra bracket $[\cdot,\cdot]_E$ on the space of sections $\bsection{E}$,
	such that $\rho_E$ induces a Lie algebra homomorphism from $\bsection{E}$ to $\mathscr{X}(M)\otimes_{\R}{\K}$, and the Leibniz rule
	$$[u,fv]_E=\big(\rho_E(u)f\big)v+f[u,v]_E$$
	is satisfied for all $f \in C^{\infty} (M,{\K})$ and $u,v \in \Gamma(E)$. Such a Lie algebroid is denoted by the triple $(E,[\cdot,\cdot]_E,\rho_E)$.
\end{Def}

\begin{Def}
	By a \textbf{Lie algebroid pair} (Lie pair for short) $(L,A)$,
	we mean a Lie algebroid $L$ together with a Lie subalgebroid $A$ of $L$ 
	over the same base manifold $M$ (we will often omit to write $M$).
	%In other words,
	%$A$ is a vector sub-bundle of $L$ such that the Lie bracket 
	%$[\cdot,\cdot]_L$ is closed when restricted to $\Gamma(A)$. 
	%Alternatively, we can regard  $$i\colon A=(A,[\cdot,\cdot]_A,\rho_A) \to %L=(L,[\cdot,\cdot]_L,\rho_L)$$ as an embedding of Lie algebroids over %$M\stackrel{\id}{\to} M$.
\end{Def}

Given a Lie pair $(L,A)$, one can choose a complement $B$ of $A$
in $L$ and identify the quotient bundle $L/A\cong B$
though it is not canonical.
In the sequel, we fix such an embedding $B\hookrightarrow L$ and
hence get a fixed decomposition $L\cong A\oplus B$.
The projections from $L$ to $A$ and $B$ are denoted by ${\rm pr}_A$ and ${\rm pr}_B$.

The \textbf{Bott connection} $\nabla$ of $A$ on $B$ is given by 
$$\Gamma(A)\times\Gamma(B)\longrightarrow \Gamma(B): \quad (a,b)\longmapsto \nabla_a b:={\rm pr}_B[a,b]_L.$$
In fact, this definition does not depend on the choice of a decomposition 
$L\cong A\oplus B$.  Therefore, the quotient $B$ of a Lie pair $(L,A)$ is 
canonically an $A$-module, i.e.,
$$\nabla_a (fb)=(\rho_A(a)f)b+f(\nabla_a b),$$
$$\nabla_{[a_1,a_2]_A}b=\nabla_{a_1}(\nabla_{a_2}b)-\nabla_{a_2}(\nabla_{a_1}b),$$
where $a,a_1,a_2\in\Gamma(A)$, $b\in\Gamma(B)$, $f\in \CinfMK $.

Similarly, there is a $B$-operation $\eth$ on $A$ though it depends on the decomposition $L\cong A\oplus B$:
$$\Gamma(B)\times\Gamma(A)\longrightarrow \Gamma(A):\quad
(b,a)\longmapsto \eth_b a:={\rm pr}_A[b,a]_L.$$
Furthermore, the operation $\piepartial_b\colon \Gamma(A)\to \Gamma(A),\forall b \in \Gamma(B)$ 
induces a dual operation on $\Gamma(A^\vee)$ given by
$$\pairing{\piepartial_b u}{a}
=\rho_L(b)\pairing{u}{a}-\pairing{u}{\piepartial_b a} ,\quad
\forall u \in \Gamma(A^\vee),\forall a \in \Gamma(A).$$
More generally, the operation $\piepartial_b$ extends to a derivation on $\Gamma(\Lambda^p A^\vee),\forall p \in \mathbb{N}$ by the Leibniz rule.

Next, introduce the following maps for any $b_1, b_2\in\Gamma(B)$:
	\begin{align*}
		[\cdot,\cdot]_B: ~\Gamma(B)\times\Gamma(B)\to\Gamma(B), \quad [b_1,b_2]_B&:={\rm pr}_B[b_1,b_2]_L,\\
		\beta(\cdot,\cdot):~ \Gamma(B)\times\Gamma(B)\to\Gamma(A), \quad \beta(b_1,b_2)&:={\rm pr}_A[b_1,b_2]_L,
	\end{align*}
where the bracket $[\cdot,\cdot]_B$ does not necessarily satisfy the Jacobi identity.

With these structure maps fixed, the Lie algebroid structure on $L=A\oplus B$ can be described as follows:
\begin{equation*}\label{Liealgebroidstructure}
	\begin{cases}
		[a_1,a_2]_L=[a_1,a_2]_A,\\
		[b_1,b_2]_L=\beta(b_1,b_2)+[b_1,b_2]_B,\\
		[a,b]_L=-\eth_b a+\nabla_a b,\\
		\rho_L(a+b)=\rho_A(a)+\rho_B(b),
	\end{cases}
	\forall a,a_1,a_2\in\Gamma(A),  b,b_1,b_2\in\Gamma(B),
\end{equation*}
where $\rho_A,\,\rho_B$ are the restrictions $\rho_L|_A,\,\rho_L|_B$ respectively.

%Introduce the following identities:
%$$[b_1,[b_2,b_3]_B]_B+[b_3,[b_1,b_2]_B]_B+[b_2,[b_3,b_1]_B]_B=
%\nabla_{\beta(b_2,b_3)}b_1+\nabla_{\beta(b_3,b_1)}b_2+\nabla_{\beta(b_1,b_2)}b_3,$$\\

Consider the spaces of $A$-forms and $B$-valued $A$-forms:
$$\Omega_A^\bullet = \bigoplus \limits_{k=0}^{\rank(A)}\Gamma(\Lambda^k A^\vee),\qquad \Omega^\bullet_A(B)=\bigoplus \limits_{k=0}^{\rank(A)}\Gamma(\Lambda^k A^\vee \otimes B).$$
Let $\omega\in\Omega^\bullet_A$ and 
$\lambda\in\Omega^\bullet_A,b\in \Gamma(B)$ so that $\lambda\otimes b \in \Omega^\bullet_A(B)$. It is clear that $\Omega^\bullet_A(B)$ is an $\Omega_A^\bullet$-module: 
$\omega \cdot (\lambda\otimes b):=(\omega \wedge \lambda) \otimes b$.
Also, both  $\Omega^\bullet_A$ and  $\Omega_A^\bullet(B)$ are equipped with the standard Chevalley-Eilenberg differentials as follows:
\begin{align}\nonumber 
	(\dA \omega)(a_1,\cdots,a_{k+1})=&\sum\limits_{i=1}^{k+1} (-1)^{i+1}
	\rho_A (a_i) \big(\omega(a_1,\cdots,\hat{a_i},\cdots,a_{k+1})\big)\\\nonumber 
	& +\sum\limits_{i<j} (-1)^{i+j} \omega([a_i,a_j]_A,a_1,\cdots,\hat{a_i},\cdots,\hat{a_j},\cdots,a_{k+1}),
\end{align}
\begin{align}\label{Eqt:dABot} 
	(\dBott   {X})(a_1,\cdots,a_{k+1})=&\sum\limits_{i=1}^{k+1} (-1)^{i+1} \nabla_{a_i}\big({X}(a_1,\cdots,\hat{a_i},\cdots,a_{k+1})\big)\\\nonumber 
	& +\sum\limits_{i<j} (-1)^{i+j} {X}([a_i,a_j]_A,a_1,\cdots,\hat{a_i},\cdots,\hat{a_j},\cdots,a_{k+1}),
\end{align}
where $\omega \in \Omega_A^k,{X} \in \Omega^k_A(B),a_i\in\Gamma(A)$.
Under these differentials, the space $(\Omega^\bullet_A,\dA )$
is a dg algebra and the space $(\Omega^\bullet_A(B),\dBott)$ is a dg   
$\Omega^\bullet_A$-module.

A vector field of degree $n$ on the graded manifold $A[1]$ is a derivation of degree $n$ of the algebra $\Omega_A^\bullet=C^\infty(A[1])$, i.e., a linear map $\varsigma\colon \Omega_A^\bullet  \rightarrow   \Omega_A^{\bullet + n} $ such that
the graded Leibniz rule
$$\varsigma(\xi \wedge \eta)=(\varsigma\xi)\wedge \eta +(-1)^{n\cdot |\xi|} \xi \wedge (\varsigma\eta)$$
holds for all homogeneous elements
$\xi,\eta \in \Omega_A^\bullet$. 

\textbf{Notation:}  Let us denote by $\Der^n (\Omega_A^\bullet )$
the set of degree $n$ derivations of $\Omega_A^\bullet$.

\begin{Thm}\label{Thm:Linftystructuregenerator}\cite[Proposition 4.3]{BCSX}
	Let $(L,A)$ be a Lie pair. Given a decomposition $L\cong A\oplus B$, there exists an induced ${\Linftythree}$ algebroid  structure on the graded vector bundle $A[1]\times B\to A[1]$ whose structure maps are given as follows:

	\begin{enumerate}
		\item[(1).]  The zeroth anchor $\rho_0$ is $\dA : \Omega^\bullet_A\to\Omega^{\bullet+1}_A$.\\
		
		\item[(2).] The unary anchor
		$\rho_1: \Omega^i_A(B)\to\Der^i(\Omega^\bullet_A)$
		is given by
		$$\rho_1(\lambda\otimes b)\omega=\lambda\cdot(\piepartial_b \omega), \qquad \qquad \forall \lambda\otimes b\in\Omega^i_A(B),  \omega\in\Omega^\bullet_A.$$
		
		\item[(3).] The binary anchor $\rho_2:\Omega^i_A(B)\times\Omega^j_A(B)\to\Der^{i+j-1}(\Omega^\bullet_A)$
		is given by
		$$\rho_2(\lambda\otimes b, \lambda'\otimes b')\omega=
		(-1)^{|\lambda|+|\lambda'|+1}
		(\lambda\wedge\lambda')\cdot(i_{\beta(b,b')}\omega), $$
		$ \forall \lambda\otimes b,\lambda'\otimes b'\in\Omega^i_A(B),  \omega\in\Omega^\bullet_A$.
		\item[(4).]  The unary bracket $d=[\cdot]_1$ is
		$\dBott\colon\Omega^\bullet_A(B)\to \Omega^{\bullet+1}_A(B)  $.\\
		\item[(5).]  The binary bracket
		$[\cdot,\cdot]_2: \Omega^i_A(B)\times\Omega^j_A(B)\to\Omega^{i+j}_A(B)$
		is generated by the relations
		\begin{align*}
			[b,b']_2&=[b,b']_B,\\
			[X,\omega\cdot Y]_2&=(\rho_1(X)\omega)\cdot  Y+(-1)^{|\omega|\cdot|X|} \omega \cdot [X,Y]_2,
		\end{align*}
	 $\forall b,b' \in \Gamma(B), X,Y \in \Omega^\bullet_A(B), \omega \in \Omega^\bullet_A$.
		\item[(6).]  The ternary bracket
		$[\cdot,\cdot,\cdot]_3:  \Omega_A^p(B) \times \Omega_A^q(B) \times \Omega_A^r(B) \rightarrow \Omega_A^{p+q+r-1}(B)$
		is generated by the relations
		\begin{align*}
			[b,b',b'']_3&=0, \\
			[X,Y,\omega \cdot  Z]_3&=(\rho_2(X,Y)\omega)\cdot  Z+(-1)^{|\omega|(|X|+|Y|+1)} \omega \cdot [X,Y,Z]_3,
		\end{align*}
	 $\forall b,b',b''\in \Gamma(B),	 X,Y,Z \in \Omega^\bullet_A(B), \omega \in \Omega^\bullet_A$.
		\item[(7).] The other higher anchors and brackets all vanish.
	\end{enumerate}
	
\end{Thm}

We further find  direct formulas of higher structure maps of $\Omega^\bullet_A(B)$. Similar results are found in \cite{Schatz-Zambon}.

\begin{prop}\label{Prop:2and3bracket} The binary and the ternary brackets   of the ${\Linftythree}$ algebra   $\Omega^\bullet_A(B)$ are expressed as follows:
\begin{itemize}\item[(1)]	
		For all $X \in \Omega_A^p(B)$ and $Y \in \Omega_A^q(B)$, the $2$-bracket $[X,Y]_2\in \Omega_A^{p+q}(B)$ is determined by
		\begin{align*}
			& \quad [X,Y]_2(a_1,\cdots, a_{p+q})\\
			&=\sum\limits_{\sigma\in\Sh(p,q)}\sum\limits_{i=1}^p 
			{\rm sgn}(\sigma)
			X(a_{\sigma(1)},\cdots, \piepartial_{Y(a_{\sigma(p+1)},\cdots,a_{\sigma(p+q)})}a_{\sigma(i)},
			\cdots, a_{\sigma(p)})\\
			&\quad -\sum\limits_{\tau\in\Sh(p,q)}\sum\limits_{j=1}^q {\rm sgn}(\tau)
			Y(a_{\tau(p+1)},\cdots,
			\piepartial_{X(a_{\tau(1)},\cdots,a_{\tau(p)})}a_{\tau(p+j)}, \cdots
			,a_{\tau(p+q)}
			)\\
			&\quad\quad + \sum \limits_{\alpha \in \Sh(p,q)} {\rm sgn}(\alpha)[X(a_{\alpha(1)},\cdots, a_{\alpha(p)}), Y(a_{\alpha(p+1)},\cdots,a_{\alpha(p+q)})]_B\,.
		\end{align*}
		\item[(2)]For all $X \in \Omega_A^p(B)$,  $Y \in \Omega_A^q(B)$, and $Z \in \Omega_A^r(B)$, the $3$-bracket $[X,Y,Z]_3\in \Omega_A^{p+q+r-1}(B)$ is determined by
		\begin{align*}
			&[X,Y,Z]_3 (a_1,\cdots,a_{p+q+r-1})\\
			&=\!(-1)^{p+q+1} \!\!\!\!\!\!\!\!\!\sum\limits_{\sigma \in\Sh(p,q,r-1)}   \!\!\!\!\!\!{\rm sgn}(\sigma)
			Z \Big(\beta \big(X(a_{\sigma(1)},\!\cdots),\!
			Y(a_{\sigma(p+1)},\!\cdots)\big),\!
			a_{\sigma(p+q+1)},\!\cdots \Big)\\
			 &\quad+\!(-1)^{p} \!\!\!\!\!\!\!\! \sum\limits_{\tau \in\Sh(p,q-1,r)}
			\!\!\!\!\!\! {\rm sgn}(\tau)
			Y\Big(\beta \big(X(a_{\tau(1)},\cdots),
			Z(a_{\tau(p+q)},\cdots)\big),
			a_{\tau(p+1)},\cdots\Big)
			\\
			&\quad-\sum \limits_{\alpha \in {\Sh(p-1,q,r)}}   
			{\rm sgn}(\alpha)
			X\Big(\beta \big(Y(a_{\alpha(p)},\cdots),
			Z(a_{\alpha(p+q)}, \cdots)\big),
			a_{\alpha(1)},\cdots\Big).
		\end{align*}
	\end{itemize}
	Here $a_1,\cdots, a_{p+q+r-1}\in\Gamma(A)$, and
	the set $\Sh(i,j,k)$ consists of $(i,j,k)$-shuffles.
\end{prop}

\begin{proof}
	To show (1), by linearity, we assume that $X=u\otimes b\in \Omega_A^p(B)$ 
	and $Y=v\otimes c\in \Omega_A^q(B)$ where $u,v\in \Omega_A^{\bullet}$, $b,c\in \Gamma(B)$.
	By the generating relations of the binary bracket, we have
	$$[X,Y]_2=\big(u\wedge(\piepartial_b v)\big)\otimes c-\big((\piepartial_c u)\wedge v\big)\otimes b
	+(u\wedge v)\otimes [b,c]_B.$$
	Evaluating the above expression at the variables $a_1,\cdots, a_{p+q}$, we get
	\begin{align*}
		&[X,Y]_2(a_1,\cdots, a_{p+q})\\
		&\quad=\sum\limits_{\tau\in\Sh(p,q)} {\rm sgn}(\tau)
		u(a_{\tau(1)},\cdots,a_{\tau(p)})\cdot
		\piepartial_b v(a_{\tau(p+1)},\cdots,a_{\tau(p+q)})\cdot c\\
		&\qquad-\sum\limits_{\sigma\in\Sh(p,q)} {\rm sgn}(\sigma)
		\piepartial_c u(a_{\sigma(1)},\cdots,a_{\sigma(p)}) \cdot  v(a_{\sigma(p+1)},\cdots,a_{\sigma(p+q)}) \cdot b\\
		&\qquad+\sum\limits_{\alpha \in \Sh(p,q)} {\rm sgn}(\alpha) u(a_{\alpha(1)},\cdots, a_{\alpha(p)}) \cdot v(a_{\alpha(p+1)},\cdots,a_{\alpha(p+q)})\cdot [b,c]_B\\
		&\quad=\sum\limits_{\tau\in\Sh(p,q)} {\rm sgn}(\tau)
		u(a_{\tau(1)},\cdots,a_{\tau(p)})\cdot
		\rho_B(b)\big( v(a_{\tau(p+1)},\cdots,a_{\tau(p+q)})\big)\cdot c\\
		&\qquad-\sum\limits_{\tau\in{\Sh(p,q)}} {\rm sgn}(\tau)
		u(a_{\tau(1)},\cdots,a_{\tau(p)})\\&\qquad\qquad\qquad\qquad\qquad\qquad\cdot
		\big(\sum\limits_{j=1}^q v(a_{\tau(p+1)},\cdots,\piepartial_b a_{\tau(p+j)},\cdots, a_{\tau(p+q)})\big)\cdot c\\
		&\qquad-\sum\limits_{\sigma\in\Sh(p,q)} {\rm sgn}(\sigma)
		\rho_B(c)\big(u(a_{\sigma(1)},\cdots,a_{\sigma(p)})\big) \cdot  v(a_{\sigma(p+1)},\cdots,a_{\sigma(p+q)}) \cdot b\\
		&\qquad+\sum\limits_{\sigma\in\Sh(p,q)} {\rm sgn}(\sigma)
		\big(\sum\limits_{i=1}^p u(a_{\sigma(1)},\cdots,\piepartial_c a_{\sigma(i)},\cdots,a_{\sigma(p)})\big)\\
		&\quad\quad\quad\qquad\qquad\qquad\qquad\qquad\qquad \cdot  v(a_{\sigma(p+1)},\cdots,a_{\sigma(p+q)}) \cdot b\\
		&\qquad+\sum\limits_{\alpha \in \Sh(p,q)} {\rm sgn}(\alpha) u(a_{\alpha(1)},\cdots, a_{\alpha(p)}) \cdot v(a_{\alpha(p+1)},\cdots,a_{\alpha(p+q)})\cdot [b,c]_B\\
		&\quad=\sum \limits_{\sigma \in \Sh(p,q)}\sum\limits_{i=1}^p {\rm sgn}(\sigma)
		X(a_{\sigma(1)},\cdots, \piepartial_{Y(a_{\sigma(p+1)},\cdots,a_{\sigma(p+q)})}a_{\sigma(i)},
		\cdots, a_{\sigma(p)})\\
		&\qquad-\sum \limits_{\tau \in \Sh(p,q)}\sum\limits_{j=1}^q {\rm sgn}(\tau)
		Y(a_{\tau(p+1)},\cdots,
		\piepartial_{X(a_{\tau(1)},\cdots,a_{\tau(p)})}a_{\tau(p+j)}, \cdots
		,a_{\tau(p+q)}
		)\\
		&\qquad+ \sum \limits_{\alpha \in \Sh(p,q)} {\rm sgn}(\alpha)[X(a_{\alpha(1)},\cdots, a_{\alpha(p)}), Y(a_{\alpha(p+1)},\cdots,a_{\alpha(p+q)})]_B\,.
	\end{align*}
	
	To show (2),  we also assume that $Z=w\otimes e\in \Omega_A^r(B)$ where $w\in \Omega_A^{\bullet}$, $e\in \Gamma(B)$.
	By the generating relations of the ternary bracket, we have
	\begin{align*}
		[X,Y,Z]_3
		&=\big(\rho_2(X,Y) w\big)\otimes e 
		+ (-1)^{|Y||Z|+1} \big(\rho_2(X,Z) v\big)\otimes c\\
		&\qquad
		+ (-1)^{|X|(|Y|+|Z|)} \big(\rho_2(Y,Z) u\big)\otimes b\\
		&=(-1)^{|X|+|Y|+1}(u\wedge v\wedge i_{\beta(b,c)}w) \otimes e
		+(-1)^{|X|}(u\wedge i_{\beta(b,e)}v\wedge w) \otimes c\\
		&\qquad
		- (i_{\beta(c,e)}u\wedge v\wedge w) \otimes b.
	\end{align*}
	Evaluating the above expression at the variables $a_1,\cdots, a_{p+q}$, we get (2).
\end{proof}

\begin{Rem}
	The unary anchor $\rho_1$, the binary anchor $\rho_2$ and the ternary bracket $[\cdot,\cdot,\cdot]_3$ are all $C^\infty(M)$-(multi-)linear, whereas the zeroth anchor $d_A$, the unary bracket $\dBott$ and the binary bracket $[\cdot,\cdot]_2$ are not.
\end{Rem}

\begin{Rem}\label{Rmk:depedenceofsplittings}Note that the $\Linftythree$ structure maps depend on the choice of a splitting $L\cong A\oplus B$. However, different choices of splittings give rise to
	isomorphic $\Linftythree$ algebras where the isomorphism is given by a collection
	of multilinear maps
	$$\varphi_n:~\wedge^n \big(\Omega^\bullet_A(B)\big)\to \Omega^\bullet_A(B),\quad n=1,2,\ldots$$
	where $\varphi_1$ is the identity map, 	see \cite[Theorem 1.1]{BCSX}. In fact, these maps $\varphi_n$ can be wrapped up into an automorphism $\exp(\delta_\pi)$ of the coalgebra ${S\big(\Omega^\bullet_A(B)[1]\big)} $ where the datum $\pi$ measures the difference between two splittings and  $\delta_\pi$ is a coderivation of $ {S\big(\Omega^\bullet_A(B)[1]\big)}$.  Moreover, $\exp(\delta_\pi)$ is an isomorphism of dg coalgebras, i.e., it  intertwines  the relevant codifferentials arising from different splittings.
\end{Rem}

\section{Lie algebra actions on $L_\infty$ algebras}\label{Sec:LiealgebraactiononLinfty}
Let $\h$ be an  ordinary (i.e., with zero grading) Lie algebra,   
and $\g$ be an $L_\infty$ algebra.
In this Section, we consider a notion of Lie algebra
action of $\h$ on the $L_\infty$ algebra $\g$,
together with several equivalent descriptions of such an action.

\subsection{Definition of Lie algebra actions on $L_\infty$ algebras}
Recall that from $\g$ we obtain the associated dg coalgebras 
$({S(\g[1])},Q)$ and $(\overline{S(\g[1])},Q)$, and the dglas  
${\rm Coder}(\overline{S(\g[1])})$ and ${\rm Coder}({S(\g[1])})$ whose  
differentials are both induced by $Q$. In the meantime, $\h$ could be regarded  
as a dgla  which concentrates in degree $0$ and has a trivial differential.
\begin{Def}\label{Def: haction} A \textbf{Lie algebra action} of $\h$ on the $L_\infty$ algebra $\g$ is a ~dgla ~ homomorphism
	$$\psi:~\h\to {\rm Coder}({S(\g[1])}).$$
\end{Def}
\begin{Def}\label{Def: hactioncompatible}
	Suppose that two $L_\infty$ algebras $\g_1$ and $\g_2 $ are both equipped with $\h$-actions $\psi_1$ and $\psi_2$, respectively. 
	A morphism of $L_\infty$ algebras $F:~\g_1 \to \g_2 $ is called \textbf{compatible with the $\h$-actions} if the corresponding morphism of dg coalgebras $F^s:~  {S(\g_1 [1])} \to   {S(\g_2 [1])} $ intertwines with $\psi_1$ and $\psi_2$, i.e. $\psi_2(h)\circ F^s=F^s\circ \psi_1(h)$, $\forall h\in \h$.  
	
\end{Def}

\begin{prop}\label{Prop:equi-action1}
	An $\h$-action on $\g$ as defined above is equivalent to
	a pair of linear maps $(\theta,\gamma)$
	where  $\theta:~\h\to {\rm Coder}^0(\overline{S(\g[1])})$
	and $\gamma:~\h\to\g[1]^0$
	%\begin{itemize}
	%	\item $\theta:~\h\to {\rm Coder}^0(\overline{S(\g[1])})$ 
	%	is a morphism of  vector spaces;
	%	\item $\gamma:~\h\to\g[1]^0$ is a morphism of vector spaces
	%\end{itemize}	
	such that for all $  h,h'\in\h$ the following compatibility conditions hold:
	\begin{eqnarray}
		\label{Eqt:Q-kappa0}
		Q\circ \gamma&=&0  ,\\
		\label{Eqt:Q-kappa}
		[Q,\theta(h)] &=&-\gamma(h)\lrcorner Q ,\\
		\label{Eqt:Q-theta0}
		\gamma\big([h,h']_\h\big)&=&\theta(h)\big(\gamma(h')\big)-\theta(h')\big(\gamma(h)\big),\\
		\label{Eqt:Q-theta}
		\theta\big([h,h']_\h\big)&=&[\theta(h),\theta(h')]+\gamma(h')\lrcorner\theta(h)-\gamma(h)\lrcorner\theta(h').
	\end{eqnarray}
\end{prop}

\begin{proof}
	For every $h\in\h$, its corresponding coderivation $\psi(h)\in {\rm Coder}({S(\g[1])})$, by identification \eqref{ISO:Fukaya}, is completely determined by a sequence of maps 
	$$\psi(h)_k:~S^k(\g[1])\to \g[1],\quad k\geqslant0.$$
	Let $\gamma:~\h\to\g[1]^0$ be given by $\gamma(h):=\psi(h)_0(1)$. 
	According to identification \eqref{ISO:reduced}, the sequence of maps $\psi(h)_k$, $k\geqslant 1$, uniquely 
	determines a coderivation of degree $0$ on $\overline{S(\g[1])}$, which we denote by $\theta(h)$. Hence,
	\begin{equation}\label{eq:decomposition}
		\psi(h)=\gamma(h)^\#+\theta(h).
	\end{equation}
	
	The linear map $\psi$ is a morphism of dglas  if and only if both of the following two equalities hold:
	\begin{align}
		[Q,\psi(h)]&=0,\quad \forall h\in\h, \label{Eqt:actionstrcture1} \\
		\psi([h,h']_\h)&=[\psi(h),\psi(h')],\quad \forall h,h'\in\h. \label{Eqt:actionstrcture2}
	\end{align}
	
	Using \eqref{eq:decomposition} and \eqref{eq:lrcorner}, we have 
	\begin{align*}
		[Q,\psi(h)] 
		&= [Q,\gamma(h)^\#] +[Q,\theta(h)] = -[\gamma(h)^\#,Q]+[Q,\theta(h)]\\
		&= Q\big(\gamma(h)\big)^\#+\gamma(h)\lrcorner Q+[Q,\theta(h)],
	\end{align*}
	where $\gamma(h)\lrcorner Q+[Q,\theta(h)]\in {\rm Coder}^0(\overline{S(\g[1])})$. Now, we see that \eqref{Eqt:actionstrcture1} is equivalent to \eqref{Eqt:Q-kappa0} and \eqref{Eqt:Q-kappa}.
	
	Using \eqref{eq:decomposition} and \eqref{eq:lrcorner} again, plus the fact $[v^\#,v'^\#]=0,\forall v,v'\in \g[1]$, we have
	\begin{align*}
		&\psi\big([h,h']_\h\big)-[\psi(h),\psi(h')]\\
		=&\gamma\big([h,h']_\h\big)^\#+\theta\big([h,h']_\h\big)
		-[\gamma(h)^\#+\theta(h),\gamma(h')^\#+\theta(h')]\\
		=&\gamma\big([h,h']_\h\big)^\#+\theta\big([h,h']_\h\big)
		-[\theta(h),\theta(h')]
		-[\gamma(h)^\#,\theta(h')]
		-[\theta(h),\gamma(h')^\#]\\
		=&\gamma\big([h,h']_\h\big)^\#+\theta\big([h,h']_\h\big)
		-[\theta(h),\theta(h')]\\
		&+\big(\theta(h')(\gamma(h))\big)^\#+\gamma(h)\lrcorner\theta(h')
		-\big(\theta(h)(\gamma(h'))\big)^\#-\gamma(h')\lrcorner\theta(h).
	\end{align*}
	Separating the coderivations in the above expression involving $\#$ apart from those in ${\rm Coder}^0(\overline{S(\g[1])})$, we see that \eqref{Eqt:actionstrcture2} is equivalent to \eqref{Eqt:Q-theta0} and \eqref{Eqt:Q-theta}.
\end{proof}

\begin{Rem}
	In particular, if $\gamma=0$, then $\theta:~\h\to {\rm Coder}(\overline{S(\g[1])})$ needs to be a morphism of dglas.
	The compatibility conditions in Proposition \ref{Prop:equi-action1} become
	\begin{eqnarray*}
		[Q,\theta(h)] &=&0 ,\\
		\theta([h,h']_\h)&=&[\theta(h),\theta(h')],
	\end{eqnarray*}
	for all $h,h'\in\h$.
	Therefore, for any $h\in\h$ the coderivation $\theta(h)$ is an infinitesimal deformation of $\g$, and $\theta$ commutes with Lie brackets.
	In this case, we say $\g$ admits a \textbf{strict Lie algebra action} of $\h$.
\end{Rem}

\subsection{Equivalent characterizations of Lie algebra actions on $L_\infty$ algebras}
We give two more characterizations of Definition~\ref{Def: haction}. The first one follows Mehta-Zambon's approach \cite{MehtaZambon} via extensions. The second one follows the classical approach of specifying the action maps.

\begin{Thm}\label{Thm:hplusg}
	An $L_\infty$ algebra $\g$ admits a Lie algebra action of $\h$
	if and only if the direct sum $\h\oplus \g$, 
	where $\h$ concentrates in degree $0$, 
	admits an $L_\infty$ algebra structure which extends the original  $L_\infty$ algebra structure on $\g$ and the Lie algebra structure on $\h$,
	in the sense that the following two conditions hold:
	\begin{itemize}
		\item[(1)] 
		%	$\g$ is an ideal of $\h\oplus\g$, 
		%	i.e. any bracket  with an input of $\g$ takes values in $\g$; 
		the sequence 
		$0\longrightarrow \g \overset{\iota_\g}{\longrightarrow} \h\oplus \g \overset{\rm{pr}_\h}{\longrightarrow} \h \longrightarrow 0$
		is a (not necessarily split) sequence of $L_\infty$ morphisms, 
		\item [(2)] any $n$-bracket on $\h\oplus \g$ for
		$n\geqslant  3$ vanishes when two or more inputs are from $\h$. 
	\end{itemize}
\end{Thm}

\begin{proof}
	Assume that $\g$ admits an $\h$-action with the dgla  homomorphism $\psi$
	as in Definition \ref{Def: haction}, it suffices to construct a codifferential $\widehat{Q}$ on $\overline{S(\h[1]\oplus\g[1])}$.
	By Proposition \ref{Prop:equi-action1}, $\psi$ is completely determined by $(\gamma,\theta)$.
	In what follows, we explain by several steps how to use $(Q,\theta,\gamma)$ to construct $\widehat{Q}$.
	
	Firstly, we specify the components of a degree $1$
	coderivation $\widehat{Q}$  $=$ $(\widehat{Q}_1,$ $\widehat{Q}_2,\cdots)$
	of $\overline{S(\h[1]\oplus\g[1])}$. 
	\begin{itemize}
		\item
		The unary component
		$\widehat{Q}_1:~\h[1]\oplus\g[1]\to\g[1]\subset\h[1]\oplus\g[1]$ is given by
		$$\widehat{Q}_1(\tilde{h})=\gamma(h)\quad\text{and}\quad \widehat{Q}_1(\tilde{x})=Q_1(\tilde{x}), \qquad \forall h\in\h, x\in\g.$$
		\item
		The binary component
		$\widehat{Q}_2:~S^2(\h[1]\oplus\g[1])\to\h[1]\oplus\g[1]$ is given by
		\begin{eqnarray*}
			\widehat{Q}_2(\tilde{x}_1\odot \tilde{x}_2) &=& Q_2(\tilde{x}_1\odot \tilde{x}_2),\\
			\widehat{Q}_2(\tilde{h}\odot \tilde{x}) &=& \theta(h)_1(\tilde{x}), 
			\\
			\widehat{Q}_2(\tilde{h}\odot \tilde{h'}) &=& [h,h']_\h [1],
		\end{eqnarray*}
		where $h,h'\in\h$ and $x,x_1,x_2\in\g$.
		\item
		The $n$-component $(n\geqslant3)$ $\widehat{Q}_n:~S^n(\h[1]\oplus\g[1])\to\g[1]\subset\h[1]\oplus\g[1]$ is determined by
		\begin{eqnarray*}
			\widehat{Q}_n(\tilde{x}_1\odot\cdots\odot \tilde{x}_n)&=&Q_n(\tilde{x}_1\odot\cdots\odot \tilde{x}_n),\\
			\widehat{Q}_n(\tilde{h}\odot \tilde{x}_1\odot\cdots\odot \tilde{x}_{n-1})&=&\theta(h)_{n-1}(\tilde{x}_1\odot\cdots\odot \tilde{x}_{n-1}),
		\end{eqnarray*}
		where $h\in\h,\,x_1,\cdots,x_n\in\g$, and $\widehat{Q}_n$ is required to vanish with two or more inputs of $\h[1]$.
	\end{itemize}
	
	Next, we verify that $\widehat{Q}\circ \widehat{Q}=0$.
	Applying Equation \eqref{Eqt:Q-kappa0} and $({Q}\circ{Q})_1=0$,
	one has $(\widehat{Q}\circ\widehat{Q})_1=0$.
	Applying Equations \eqref{Eqt:Q-kappa},\,\eqref{Eqt:Q-theta0}
	and $({Q}\circ{Q})_2=0$, one has $(\widehat{Q}\circ\widehat{Q})_2=0$.
	We now apply Equations \eqref{Eqt:Q-kappa} and \eqref{Eqt:Q-theta} to obtain $(\widehat{Q}\circ\widehat{Q})_{n+1}=0$ for every $n\geqslant 2$. 
	\begin{itemize}
		\item 
		When there is only one input of $\h[1]$, i.e.,
		$\forall\tilde{h}\in\h[1]$ and $\tilde{x}_1,\cdots,\tilde{x}_n\in\g[1]$, we have
		\begin{align*}
			&(\widehat{Q}\circ\widehat{Q})_{n+1}(\tilde{h}\odot \tilde{x}_1\odot\cdots\odot \tilde{x}_n)\\
			&\quad=Q_{n+1}(\gamma(h)\odot \tilde{x}_1\odot\cdots\odot \tilde{x}_n)\\
			&\qquad+\sum\limits_{{\begin{subarray}{c}
						k=1,\cdots,n \\
						\sigma \in \Sh(k,n-k)\end{subarray}}}
			\epsilon(\sigma;\tilde{x}_1,\cdots,\tilde{x}_n)\\
			&\qquad\qquad Q_{n-k+1}
			\Big(
			\theta(h)_k(\tilde{x}_{\sigma(1)}\odot\cdots\odot \tilde{x}_{\sigma(k)})\odot \tilde{x}_{\sigma(k+1)}\odot\cdots\odot \tilde{x}_{\sigma(n)}
			\Big)
			\\
			&\qquad-\sum\limits_{{\begin{subarray}{c}
						k=1,\cdots,n \\
						\sigma \in \Sh(k,n-k)\end{subarray}}}
			\epsilon(\sigma;\tilde{x}_1,\cdots,\tilde{x}_n)\\
			&\qquad\qquad\theta(h)_{n-k+1}
			\Big(
			Q_k(\tilde{x}_{\sigma(1)}\odot\cdots\odot \tilde{x}_{\sigma(k)})\odot \tilde{x}_{\sigma(k+1)}\odot\cdots\odot \tilde{x}_{\sigma(n)}
			\Big)
			\\
			&\quad=\big(\gamma(h)\lrcorner Q+[Q,\theta(h)]\big)_n(\tilde{x}_1\odot\cdots\odot \tilde{x}_n)\\
			&\quad=0.
		\end{align*}
		\item
		When there are two inputs of $\h[1]$, i.e.,
		$\forall\tilde{h},\tilde{h'}\in\h[1]$ and $\tilde{x}_1,\cdots,\tilde{x}_{n-1}\in\g[1]$, we have
		\begin{align*}
			&(\widehat{Q}\circ\widehat{Q})_{n+1}(\tilde{h}\odot \tilde{h'}\odot \tilde{x}_1\odot\cdots\odot \tilde{x}_{n-1})\\
			&\quad=\theta(h')_{n}\big(\gamma(h)\odot \tilde{x}_1\odot\cdots\odot \tilde{x}_{n-1}\big)
			-\theta(h)_{n}\big(\gamma(h')\odot \tilde{x}_1\odot\cdots\odot \tilde{x}_{n-1}\big)\\
			&\qquad+\theta\big([h,h']_\h\big)_{n-1}(\tilde{x}_1\odot\cdots\odot \tilde{x}_{n-1})\\
			&\qquad+\sum\limits_{{\begin{subarray}{c}
						k=1,\cdots,n-1 \\
						\sigma \in \Sh(k,n-k-1)\end{subarray}}}\epsilon(\sigma;\tilde{x}_1,\cdots,\tilde{x}_{n-1})\\
			&\qquad\qquad\theta(h')_{n-k}\big(\theta(h)_k(\tilde{x}_{\sigma(1)}\odot\cdots\odot \tilde{x}_{\sigma(k)})\odot \tilde{x}_{\sigma(k+1)}\odot\cdots\odot \tilde{x}_{\sigma(n-1)}\big)\\
			&\qquad-\sum\limits_{{\begin{subarray}{c}
						k=1,\cdots,n-1 \\
						\sigma \in \Sh(k,n-k-1)\end{subarray}}}\epsilon(\sigma;\tilde{x}_1,\cdots,\tilde{x}_{n-1})\\
			&\qquad\qquad\theta(h)_{n-k}\big(\theta(h')_k(\tilde{x}_{\sigma(1)}\odot\cdots\odot \tilde{x}_{\sigma(k)})\odot \tilde{x}_{\sigma(k+1)}\odot\cdots\odot \tilde{x}_{\sigma(n-1)}\big)\\
			&\quad=\big(\gamma(h)\lrcorner \theta(h')-\gamma(h')\lrcorner \theta(h)+\theta([h,h']_\h)\\
			&\qquad-[\theta(h),\theta(h')]\big)_{n-1}(\tilde{x}_1\odot\cdots\odot \tilde{x}_n)\\
			&\quad=0.
		\end{align*}
		\item 
		When there are exactly three inputs of $\h[1]$ and no inputs of $\g[1]$,
		the equality $(\widehat{Q}\circ\widehat{Q})_{3}(\tilde{h},\tilde{h'},\tilde{h''})=0$ 
		follows from the Jacobi identity for $\h$.
		\item 
		The remaining case of verifying $(\widehat{Q}\circ\widehat{Q})_{n+1}=0$
		follows from the requirement that $\widehat{Q}_{k+1}$ vanishes with two or more inputs of $\h[1]$ for $k\geqslant 2$.
	\end{itemize}
	Thus $\widehat{Q}$ is a codifferential on $\overline{S(\h[1]\oplus\g[1])}$. 
	The verification of conditions (1) and (2) are immediate.

	Conversely, suppose that $\widehat{Q}$ is a codifferential  on $\overline{S(\h[1]\oplus\g[1])}$ satisfying conditions (1) and (2) of the current Theorem.
	Define $$\gamma(h):=\widehat{Q}_1(\tilde{h}),\quad \text{and}\quad\theta(h):=\mbox{restriction of } \tilde{h}\lrcorner \widehat{Q} \mbox{ on }{\overline{S(\g[1])}},\quad \forall h\in\h, 
	$$
	where $\widehat{Q}_1(\tilde{h})$ is in $\g[1]$ because of the trivial differential on $\h$.
	
	Reversing the above argument, one obtains the
	four compatibility Equations \eqref{Eqt:Q-kappa0},\,\eqref{Eqt:Q-kappa},\,\eqref{Eqt:Q-theta0},\,\eqref{Eqt:Q-theta}
	for $ \gamma$ and $\theta$ as a result of evaluating 
	$\widehat{Q}\circ\widehat{Q}=0$ for various cases of inputs.
\end{proof}
\begin{Rem}We have defined a compatibility condition between Lie algebra actions and morphism of $\Linfty$ algebras in Definition \ref{Def: hactioncompatible}. In terms of the characterization as described by the above theorem,  the compatible condition is equivalent to the statement that the morphism  $F:~\g_1 \to \g_2 $ extends to the morphism  $F^e:~\h\oplus \g_1 \to \h\oplus \g_2 $ (of $\Linfty$ algebras) with  $F^e_1=\id_{\h}\oplus F_1$  and $F^e_{n}=F_n$ for all $n\geqslant 2$. 
	
\end{Rem}

Classically, a Lie algebra action of $\h$ on an object $\mathcal{M}$ is specified by an action map $\h\times \mathcal{M}\rightarrow \mathcal{M}$ satisfying certain compatibility conditions. Similarly, applying $(\theta,\gamma)$ (see Proposition \ref{Prop:equi-action1}) to define for $n\geqslant 1$:
\begin{equation*}\label{Eqt:mudefine}
	\mu_{n}(h,x_1,\cdots,x_{n})[1]:=(-1)^{\frac{(n+1)(n+2)}{2}+\sum_{i=1}^n(n-i)\degree{x_i}}\theta(h)_{n}(\tilde{x}_1\odot\cdots\odot\tilde{x}_n),
\end{equation*}
and $\mu_0(h)[1]:=-\gamma(h)$,
where $h\in\h$ and $x_1,\cdots,x_n\in\g$,
we have the following statement.
\begin{prop}\label{Prop:hactiontomun}
	Let $\h$ be a Lie algebra and $(\g,[\cdots]_k)$ an $L_\infty$ algebra as defined in Definition \ref{Def:Linfty-algebra}. An $\h$-action on $\g$ as defined above is equivalent to	a collection of multilinear maps $\mu_n\colon \h\times\wedge^{n}\g\to\g$ of degree $(1-n)$  for $n\geqslant 0$ which satisfy 
	\smallskip
	\begin{equation}
	\begin{aligned}\label{Eqt:mu-1}
			&\sum\limits_{{\begin{subarray}{c}
						p=1,\cdots,n\\
						\sigma \in \Sh(p,n-p)\end{subarray}}}\chi(\sigma;x_1,\cdots,x_n)\mu_{n-p+1}\big(h,[x_{\sigma(1)} ,\cdots,x_{\sigma(p)}]_p,x_{\sigma(p+1)},\cdots,x_{\sigma(n)}\big)\\
			&\qquad =\sum\limits_{{\begin{subarray}{c}
						p=0,\cdots,n \\
						\sigma \in \Sh(p,n-p)\end{subarray}}}(-1)^{p+1}\chi(\sigma;x_1,\cdots,x_n)\\
					&\qquad\qquad\qquad\qquad\qquad\qquad\big[\mu_{p}(h,x_{\sigma(1)} ,\cdots,x_{\sigma(p)}),x_{\sigma(p+1)},\cdots,x_{\sigma(n)}\big]_{n-p+1},
	\end{aligned}\end{equation}
	and 
		
	\begin{equation}\label{Eqt:mu-2}
		\begin{aligned}
			&\mu_{n}([h,h']_\h,x_1,\cdots,x_n)\\
			&\quad=\sum\limits_{{\begin{subarray}{c}
						p=0,\cdots,n \\
						\sigma \in \Sh(p,n-p)\end{subarray}}}\chi(\sigma;x_1,\cdots,x_n)\\
					&\qquad\qquad\qquad\qquad\mu_{n-p+1}\big(h,\mu_{p}(h',x_{\sigma(1)} ,\cdots,x_{\sigma(p)}),x_{\sigma(p+1)},\cdots,x_{\sigma(n)}\big)\\
					&\\
			&\qquad-\sum\limits_{{\begin{subarray}{c}
						p=0,\cdots,n \\
						\sigma \in \Sh(p,n-p)\end{subarray}}}\chi(\sigma;x_1,\cdots,x_n)\\
					&\qquad\qquad\qquad\qquad\mu_{n-p+1}\big(h',\mu_{p}(h,x_{\sigma(1)} ,\cdots,x_{\sigma(p)}),x_{\sigma(p+1)},\cdots,x_{\sigma(n)}\big),
		\end{aligned}
	\end{equation} for all $h,h'\in\h$ and $x_1,\cdots,x_n\in\g$.
\end{prop}

\begin{proof}
	When $n=0$, one immediately sees that Equations \eqref{Eqt:mu-1} and \eqref{Eqt:mu-2} are equivalent to Equations \eqref{Eqt:Q-kappa0} and \eqref{Eqt:Q-theta0}. Now assume $n\geqslant 1$, we will verify that Equation \eqref{Eqt:mu-1} is equivalent to Equation \eqref{Eqt:Q-kappa}. 
	
	By the correspondence between codifferentials $Q$ on $\overline{(S(\g[1]))}$ and $L_\infty$ brackets $[\cdots]_k,k\geqslant 1$ on $\g$, we know that 
	\begin{equation*}
		Q_n(\tilde{x}_1,\cdots,\tilde{x}_n)=(-1)^{\frac{n(n+1)}{2}+\sum_{i=1}^n(n-i)\degree{x_i}}([x_1,\cdots,x_n]_n)[1].
	\end{equation*}
	
	Equation \eqref{Eqt:Q-kappa} means that $$Q_{n+1}(\gamma(h)\odot \tilde{x}_1\odot\cdots\odot \tilde{x}_n)+[Q,\theta(h)]_n(\tilde{x}_1\odot\cdots\odot \tilde{x}_n)=0.$$
	We have the explicit evaluations:
		\begin{align*}
			&(Q\circ\theta(h))_n(\tilde{x}_1\odot\cdots\odot\tilde{x}_n)\\
			&\qquad=\sum\limits_{{\begin{subarray}{c}
						p=1,\cdots,n \\
						\sigma \in \Sh(p,n-p)\end{subarray}}}\epsilon(\sigma;\tilde{x}_1,\cdots, \tilde{x}_n)\\
			&\qquad\qquad\qquad Q_{n-p+1}\big(
			\theta(h)_p(\tilde{x}_{\sigma(1)}\odot\cdots\odot\tilde{x}_{\sigma(p)})\odot\tilde{x}_{\sigma(p+1)}\odot\cdots\odot\tilde{x}_{\sigma(n)}
			\big)\\
			&\\
			&\qquad=\sum\limits_{{\begin{subarray}{c}
						p=1,\cdots,n \\
						\sigma \in \Sh(p,n-p)\end{subarray}}}\Big(\epsilon(\sigma;\tilde{x}_1,\cdots, \tilde{x}_n)(-1)^{\frac{(p+1)(p+2)}{2}+\sum_{i=1}^p(p-i)\degree{x_{\sigma(i)}}}\\
			&\qquad\qquad\qquad \cdot Q_{n-p+1}\big(
			\mu_{p}(h,x_{\sigma(1)},\cdots,x_{\sigma(p)})[1]\odot\tilde{x}_{\sigma(p+1)}\odot\cdots\odot\tilde{x}_{\sigma(n)}
			\big)\Big)\\
			&\qquad=\sum\limits_{{\begin{subarray}{c}
						p=1,\cdots,n \\
						\sigma \in \Sh(p,n-p)\end{subarray}}}\Big(\epsilon(\sigma;\tilde{x}_1,\cdots, \tilde{x}_n)(-1)^{\frac{(p+1)(p+2)}{2}+\sum_{i=1}^p(p-i)\degree{x_{\sigma(i)}}}\\
			&\qquad\qquad (-1)^{\frac{(n-p+1)(n-p+2)}{2}+(n-p)(\degree{x_{\sigma(1)}}+\cdots\degree{x_{\sigma(p)}}+2-p-1)+\sum_{j=1}^{n-p}(n-p-j)\degree{x_{\sigma(p+j)}}}\\
			&\qquad\qquad \qquad\cdot\big[
			\mu_{p}(h,x_{\sigma(1)},\cdots,x_{\sigma(p)}),{x}_{\sigma(p+1)},\cdots,{x}_{\sigma(n)}
			\big]_{n-p+1}[1]\Big)\\
			&\qquad=\sum\limits_{{\begin{subarray}{c}
						p=1,\cdots,n \\
						\sigma \in \Sh(p,n-p)\end{subarray}}}\Big(\chi(\sigma;{x}_1,\cdots, {x}_n)(-1)^{\frac{n(n+1)}{2}+p+\sum_{i=1}^n(n-i)\degree{x_i}}\\
			&\qquad\qquad \cdot\big[
			\mu_{p}(h,x_{\sigma(1)},\cdots,x_{\sigma(p)}),{x}_{\sigma(p+1)},\cdots,{x}_{\sigma(n)}
			\big]_{n-p+1}[1]\Big).
			\end{align*}
	Likewise, one verifies two more explicit evaluations:
		\begin{align*}
			&Q_{n+1}(\gamma(h)\odot \tilde{x}_1\odot\cdots\odot \tilde{x}_n)\\
			&\qquad=(-1)^{\frac{(n+1)(n+2)}{2}+n+\sum_{i=1}^n(n-i)\degree{x_i}}[-\mu_0(h),x_1,\cdots,x_n]_{n+1}[1].\\
			&(\theta(h)\circ Q)_n(\tilde{x}_1\odot\cdots\odot\tilde{x}_n)\\
			&\qquad=\sum\limits_{{\begin{subarray}{c}
						p=1,\cdots,n \\
						\sigma \in \Sh(p,n-p)\end{subarray}}}\Big(\chi(\sigma;{x}_1,\cdots, {x}_n)(-1)^{\frac{n(n+1)}{2}+1+\sum_{i=1}^n(n-i)\degree{x_i}}\\
			&\qquad\qquad \cdot\mu_{n-p+1}\big(h,[x_{\sigma(1)} ,\cdots,x_{\sigma(p)}]_p,x_{\sigma(p+1)},\cdots,x_{\sigma(n)}\big)[1]\Big).
		\end{align*}
	Combining the above identities, we obtain the equivalence between Equations \eqref{Eqt:Q-kappa} and \eqref{Eqt:mu-1}. 
	
	Similarly, one verifies that Equation \eqref{Eqt:mu-2} for $n\geqslant 1$ is equivalent to Equation \eqref{Eqt:Q-theta}.
\end{proof}

\textbf{Notation:} In the sequel, we will simply  denote the unary map $\mu_0\colon \h\to \g^1$ by $\kappa$, 
and denote $\mu_{n} (h,x_1,\cdots,x_n)$ by $h\triangleright(x_1,\cdots,x_n)$.
We shall call them the \textbf{$n$-action maps} of $\h$ on $\g$.
When $n=1$, the symbol $h\triangleright(x_1)$ becomes $h\triangleright x_1$ if there is no risk of confusion. 

For small numbers $n$, the  compatibility conditions in  the above proposition are unraveled as follows:
\begin{itemize}
	\item $n=0$  \begin{equation}\label{Eqt:muncondition-0}
		d\circ \kappa=0;\end{equation}
	\item $n=1$ \begin{equation}\label{Eqt:muncondition-11}
		\kappa[h,h']_\h= h\triangleright(\kappa h')- {h'}\triangleright(\kappa h),\end{equation}
	\begin{equation}\label{Eqt:muncondition-12}
		d(h\triangleright x)=[\kappa h,x]_2+ h\triangleright(dx);\end{equation}
	\item $n=2$
	\begin{equation}\label{Eqt:muncondition-21}
		[h,h']_\h\triangleright  x =h\triangleright (h'\triangleright x)-h'\triangleright (h\triangleright x)+h\triangleright(\kappa h',x)-h'\triangleright(\kappa h,x),
	\end{equation}
	\begin{align}\label{Eqt:muncondition-21(2)}
		&h\triangleright(dx_1,x_2)+h\triangleright(x_1,dx_2 )+h\triangleright[x_1,x_2]_2\\
		&=-[\kappa h,x_1,x_2]_3+[h\triangleright x_1,x_2]_2+[x_1,h\triangleright x_2 ]_2-d\big(h\triangleright(x_1,x_2)\big);\nonumber
	\end{align}
	
	\item $n=3$ 
	
	\begin{equation}\begin{aligned} 
			&\sum_{\sigma\in\Sh(1,2)}\chi(\sigma;x_1,x_2,x_3)h\triangleright\big(dx_{\sigma(1)},x_{\sigma(2)},x_{\sigma(3)}\big)+h\triangleright [x_1,x_2,x_3]_3\\
			&\quad+\sum_{\sigma\in\Sh(2,1)}\chi(\sigma;x_1,x_2,x_3)h\triangleright\big([x_{\sigma(1)},x_{\sigma(2)}]_2,x_{\sigma(3)}\big)
			\\
			& =-[\kappa h,x_1,x_2,x_3]_4+\sum_{\sigma\in\Sh(1,2)}\chi(\sigma;x_1,x_2,x_3)[h\triangleright x_{\sigma(1)},x_{\sigma(2)},x_{\sigma(3)}]_{3}\\
			&\quad-\sum_{\sigma\in\Sh(2,1)}\chi(\sigma;x_1,x_2,x_3)\big[h\triangleright(x_{\sigma(1)},x_{\sigma(2)}),x_{\sigma(3)}\big]_{2}\\
			&\quad+d\big(h\triangleright(x_1,x_2,x_3)\big),
	\end{aligned}\end{equation}
	
	\begin{equation}\begin{aligned}\label{Eqt:muncondition-31}
			&[h,h']_\h\triangleright(x_1,x_2,x_3)\\
			&\quad=h\triangleright\big(\kappa h',x_1,x_2,x_3\big)\\
			&\qquad+\sum_{\sigma\in\Sh(1,2)}\chi(\sigma;x_1,x_2,x_3)h\triangleright (h'\triangleright x_{\sigma(1)},x_{\sigma(2)},x_{\sigma(3)})\\
			&\qquad+\sum_{\sigma\in\Sh(2,1)}\chi(\sigma;x_1,x_2,x_3)h\triangleright\big(h'\triangleright(x_{\sigma(1)},x_{\sigma(2)}),x_{\sigma(3)}\big)\\
			&\qquad+h\triangleright\big(h'\triangleright(x_1,x_2,x_3)\big)-h'\triangleright\big(\kappa h,x_1,x_2,x_3\big)\\
			&\qquad-\sum_{\sigma\in\Sh(1,2)}\chi(\sigma;x_1,x_2,x_3)h'\triangleright (h\triangleright x_{\sigma(1)},x_{\sigma(2)},x_{\sigma(3)})\\
			&\qquad-\sum_{\sigma\in\Sh(2,1)}\chi(\sigma;x_1,x_2,x_3)h'\triangleright\big(h\triangleright(x_{\sigma(1)},x_{\sigma(2)}),x_{\sigma(3)}\big)\\
			&\qquad-h'\triangleright\big(h\triangleright(x_1,x_2,x_3)\big).
	\end{aligned}\end{equation}

\end{itemize}

The following fact can be verified by the Jacobi identities in \eqref{eq:3Jacobi} and \eqref{Eqt:muncondition-21(2)}.

\begin{prop}\label{Prop:hactionpassingtoH}
	Suppose that  a Lie algebra $\h$ acts on  an $L_\infty$ algebra $\g$, the $0$-action being $\kappa\colon \h\to \g^1$.
	%Denote by $H(\g)$ the cohomology of $\g$
	%under the differential $d=[\cdot]_1$.
	Then  $\h_0=\ker(\kappa)$ is a Lie subalgebra in $\h$, and $\h_0$ acts on the space  $H(\g)$ (see Remark \ref{rmk:H(g)}), denoted and defined by
	$
	\triangleright:~ \h_0\times H(\g)\to H(\g)
	$,  
	\[h \triangleright \overline{x}= \overline{ h \triangleright x}\]
	for all $h\in \h_0$ and $x\in \g$ which is subject to $dx=0$. Here $\overline{x}$ stands for the cohomology class of $x$.  Moreover, the $\h_0$ action is compatible with the graded Lie algebra structure on $H(\g)$:
	\[
	h\triangleright[\overline{x}_1,\overline{x}_2]  
	=   [h\triangleright \overline{x}_1,\overline{x}_2] +[\overline{x}_1,h\triangleright \overline{x}_2] 
	\]
	for all $\overline{x}_1$ and $\overline{x}_2\in H(\g)$ (in other words, the action is a derivation).
\end{prop}

\section{Main result:  action of derivations  on the $\Linftythree $ algebra arising from a Lie pair}\label{Sec:Mainresults}

We shall show that the Lie algebra of derivations of  a Lie algebroid 
(also known as morphic vector fields or Lie algebroid derivations, see  \cite{Mackenzie-Xu,Iglesias-Laurent-Xu}) acts on the 
$L_\infty$ algebra $\Omega^\bullet_A(B)$ in the sense of 
Definition~\ref{Def: haction}. Let us first give a conceptually easy  but not very rigorous explanation of this fact. 

As we have mentioned in the introduction, the Lie algebroid structure of $L$ is encapsulated in the dg algebra $(\bsection{\wedge^\bullet L^\vee},d_L)$. In the language of  Va\u{\i}ntrob \cite{Vaintrob1997}, we say that $L[1]$ is a dg manifold. 
The group $\mathrm{Aut}(L)$ of automorphisms of the Lie 
algebroid $L$ has an induced action on $L[1]$.  According to \cite{BCSX}, the ${\Linftythree}$ algebra $\Omega^\bullet_A(B)$ can be considered as the section space of the dg vector bundle $L[1]\to A[1]$. Hence there is an associated action of $\mathrm{Aut}(L)$ on $\Omega^\bullet_A(B)$.  
Therefore,   $\Derivations(L)$, the `Lie algebra' of $\mathrm{Aut}(L)$, acts on 
$\Omega^\bullet_A(B)$ as well. Translating this thought into a purely algebraic description, it becomes  the desired action. Our  Theorem \ref{MainTheorem} below demonstrates the said action.
\begin{Def}\label{Def:Liealgebroiddiff}
	Let $(L,[\cdot,\cdot]_L,\rho_L)$ be a Lie algebroid over $M$. 
	A \textbf{derivation} of $L$ is an operator 
	$\deltap\colon\Gamma(L)\to \Gamma(L)$ which is equipped with some $\deltas\in \mathscr{X}(M)\otimes \K$, called the symbol of $\deltap$, such that
	\begin{eqnarray*}
		\label{Eqt:delta0delta1derivatives}
		\deltap(fu)&=&\deltas  ( f)  u+f\deltap  ( u) ,\\
		\label{Eqt:kdifferential3}
		[\deltas,  \rho_L(u)]&=&\rho_L(\deltap (u))   ,\\
		\label{Eqt:kdifferential4}
		\mbox{	and }~\quad	\deltap  [u,v]_L &=&[\deltap (u),v]_L +[u,\deltap  (v)]_L ,
	\end{eqnarray*}
	for all $f \in {\CinfMK}$, $u,v\in \Gamma(L)$.
\end{Def}
\textbf{Notation:} We will denote by $\Derivations(L)$   the space of derivations of a Lie algebroid $L$. Note that it is naturally a Lie algebra (possibly infinite dimensional) whose Lie bracket is the standard commutator. In \cite{AAC2012}, such derivations are called $1$-differentials of $L$.

%Our main result is the following theorem.
\begin{Thm}  \label{MainTheorem}
	Given a Lie pair $(L,A)$ and a decomposition $L\cong A\oplus B$,
	the Lie algebra $\Derivations(L)$ acts on the associated ${\Linftythree}$ 
	algebra $\Omega^\bullet_A(B)$ with the action maps specified as follows.
	\begin{itemize}
		\item[(1)] The $0$-action $\mu_0=\kappa:~\Derivations(L)\to\Omega_A^1(B)$ is given by 
		$$\kappa(\delta)(a):= -\projection_B\delta(a),\quad\forall a\in \Gamma(A).  $$
		\item[(2)] The $1$-action  $
		\mu_1:~\Derivations(L)\times \Omega_A^k(B)\to\Omega_A^k(B) $
		% ($k\geqslant  0$)
		is defined in two situations: 
		\begin{enumerate}
			\item[$\bullet$ ]if $k=0$,  then we define $\mu_1(\delta,b)=\delta\triangleright b \in \Gamma(B)$ for $\delta\in\Derivations(L)$ and $b\in \Gamma(B)$   by
			$$\delta\triangleright b :=\projection_B\delta(b);$$
			\item[$\bullet$ ]if $k\geqslant  1$, then  we define $\mu_1(\delta,X)=\delta\triangleright X \in \Omega_A^k(B)$  for 
			$\delta\in\Derivations(L)$ and $X\in\Omega_A^k(B)$   by
			$$ \begin{aligned}
				&(\delta\triangleright X) (a_1,\cdots,a_k)\\
				&\qquad:=-\sum_{j=1}^{k}X\big(a_1,\cdots,\projection_{A}\delta(a_j),\cdots,a_k\big) +\projection_{B} \delta\big( X(a_1,\cdots,a_k)\big), \end{aligned}$$ 
			where $a_1,\cdots,a_k\in\Gamma(A)$.
		\end{enumerate}

		\item[(3)] The $2$-action $
		\mu_2:~\Derivations(L)\times \Omega_A^i(B)\times \Omega_A^j(B)\to\Omega_A^{i+j-1}(B)$
		is defined by the situations:
		\begin{enumerate}
			\item[$\bullet$]If $i=j=0$, we set 
			$$\mu_2(\delta,b,b')=\delta\triangleright (b,b'):=  0,\quad \forall \delta\in \Derivations(L), b,b'\in \Gamma(B);$$
			\item[$\bullet$]If $i\geqslant  1$, $j=0$, we define 
			$\mu_2(\delta,X,b)=\delta\triangleright (X,b)\in \Omega_A^{i-1}(B)$ for $\delta\in \Derivations(L)$, $X \in \Omega_A^{i}(B)$ and $b\in \Gamma(B)$ by		
			$$\big(\delta\triangleright (X,b)\big)(a_1,\cdots,a_{i-1}):=X \big(\projection_A\delta(b),a_1,\cdots,a_{i-1}\big),$$ 
			$\forall a_1,\cdots, a_{i-1}\in \Gamma(A)$.
			\item[$\bullet$] If $i=0$ and $j\geqslant  1$, the situation  is  similar: 
			$$\big(\delta\triangleright (b,X)\big)(a_1,\cdots,a_{i-1}):= -  X \big(\projection_A\delta(b),a_1,\cdots,a_{i-1}\big),$$ 
			$\forall a_1,\cdots, a_{i-1}\in \Gamma(A)$.
			\item[$\bullet$]If $i\geqslant  1$, $j\geqslant  1$, we define $\mu_2(\delta,X,Y)=\delta\triangleright (X,Y)\in \Omega_A^{i+j-1}(B)$ for $\delta\in \Derivations(L)$, $X\in \Omega_A^{i}(B)$ and $Y\in \Omega_A^{j}(B)$ by
			 \begin{align*}
				&\big(\delta\triangleright (X,Y)\big)(a_1,\cdots,a_{i+j-1})\\
				&\quad:=(-1)^{i+1}\sum_{\sigma\in \Sh(i,j-1)}{\rm sgn}
				(\sigma)Y\Big( \projection_{A} \delta \big(X(a_{\sigma(1)},\cdots)\big) ,a_{\sigma(i+1)},\cdots\Big)\\
				&\quad\quad+\sum_{\sigma\in \Sh(i-1,j)}{\rm sgn}
				(\sigma)X \Big( \projection_{A}\delta \big(Y(a_{\sigma(i)},\cdots)\big) ,a_{\sigma(1)},\cdots\Big), \end{align*}
			where $a_1,\cdots,a_{i+j-1}\in\Gamma(A)$.
			
		\end{enumerate}
		\item [(4)]All higher $n$-actions ($n\geqslant  3$) are trivial.
	\end{itemize}
	Moreover, the structure maps  $\mu_n$ are subject to the following properties: for any $\delta\in  \Derivations(L)$ whose symbol is denoted by $\deltas$, and for any $X,Y\in \Omega_A^{\bullet}(B)$,   $f\in C^\infty(M,\K)$, $\omega\in \Omega_A^{\bullet}$, we have
	\begin{itemize}
		\item[(i)] for the $0$-action $\kappa$,
		$$\kappa(f\delta)=f\kappa(\delta);$$
		\item[(ii)] for the $1$-action,
		$$ (f\delta)\triangleright X=f \big(\delta \triangleright X\big) ;$$
		$$  \delta \triangleright (\omega\cdot X)=\big(\varrho_1(\delta)(\omega)\big)\cdot X + \omega\cdot (\delta \triangleright X) ,$$
		where $\varrho_1:~\Derivations(L)\to{\rm Der}^0(\Omega_A^\bullet)$ is defined for $\omega\in \Omega_A^k$ by
		$$
		\begin{aligned}
			&\big(\varrho_1(\delta)(\omega)\big)(a_1,\cdots,a_k)\\
			&\qquad:=
			s \big(\omega(a_1,\cdots,a_k)\big)-\sum_{j=1}^{k}\omega\big(a_1,\cdots,\projection_{A}\delta(a_j),\cdots,a_k\big);
		\end{aligned} $$
		\item[(iii)] for the $2$-action,
		%The $3$-action $\mu_3$ is generated by the following relation
		%$$\mu_3(\delta,b,b')=0$$
		%such that the following compatibility condition is satisfied: 
		$$ (f\delta)\triangleright (X,Y)=f\big(\delta \triangleright (X,Y)\big) ;$$
		$$ \delta\triangleright(  X ,\omega\cdot Y)=
		\big(\varrho_2(\delta,X)(\omega)\big)\cdot Y+(-1)^{\degree{\omega}(1+\degree{X})}\omega\cdot \big(\delta\triangleright(  X,Y)\big),$$
		where $\varrho_2:~\Derivations(L)\times \Omega_A^\bullet(B)\to {\rm Der}^{\bullet-1}(\Omega_A^\bullet)$ is defined for $X\in \Omega_A^i(B)$, $\omega\in \Omega_A^k$ by:
		\begin{eqnarray*}
			&&\big(\varrho_2(\delta,X)(\omega)\big)(a_1,\cdots,a_{i+k-1})\\
			&:=&(-1)^{i+1}\sum_{\sigma\in\Sh(i,k-1)}{\rm sgn}(\sigma)\omega\Big(\projectionA\delta\big(  X(a_{\sigma(1)},\cdots)\big),a_{\sigma(i+1)},\cdots\Big).
		\end{eqnarray*}
	\end{itemize}	
	
\end{Thm}

\begin{proof}
	Properties $(i)$, $(ii)$ and $(iii)$ follow by direct verifications
	using their definitions given by the theorem. 
	Next, we apply Proposition \ref{Prop:hactiontomun} to prove that
	these action maps $\mu_{n}$ do define a Lie algebra action. 
	Since $\mu_{k}$ and $[\cdots]_{k+1}$ vanish for
	$k\geqslant 3$, we only need to verify Equations \eqref{Eqt:mu-1} and \eqref{Eqt:mu-2} for $n=0,1,2,3$, which are unraveled in Equations \eqref{Eqt:muncondition-0} to \eqref{Eqt:muncondition-31}, plus one more equation of \eqref{Eqt:mu-1} for $n=4$.
	
	\begin{itemize}
		\item When $n=0$, we need to show $\dBott\circ \kappa$ vanishes.
		In fact, by Equation \eqref{Eqt:dABot} and the definition of $\kappa$,
		we have for $\delta\in\Derivations(L)$ and $a_1,a_2\in\Gamma(A)$ that
		\begin{align*}
			&\dBott\big(\kappa(\delta)\big)(a_1,a_2)\\
			&\quad=\nabla_{a_1}\big(\kappa(\delta)(a_2)\big)-\nabla_{a_2}\big(\kappa(\delta)(a_1)\big)-\kappa(\delta)([a_1,a_2]_A)\\
			&\quad=-\projectionB[a_1,\delta(a_2)]_L-\projectionB[\delta(a_1),a_2]_L+\projectionB\big(\delta[a_1,a_2]_A\big)=0.
		\end{align*}	
		
		\item When $n=1$, we need to verify the following two equalities for
		$\delta,\delta'\in\Derivations(L)$ and ${X}\in\Omega^\bullet_A(B)$:
		\begin{gather}
			\label{Eqt:temp1} \kappa([\delta,\delta'])=\delta\triangleright\big(\kappa(\delta')\big)-\delta'\triangleright\big(\kappa(\delta)\big),\\
			\label{Eqt:temp2} 
			\dBott\big(\delta\triangleright{X}\big)=[\kappa(\delta),{X}]_2+\delta\triangleright(\dBott {X}).
		\end{gather}
		
		To verify Equation \eqref{Eqt:temp1}, by the definition of $\kappa$, we have 
		$$\kappa(\delta)=-\projectionB \circ \delta|_{\Gamma(A)}=-\delta|_{\Gamma(A)}+\projectionA \circ \delta|_{\Gamma(A)}.$$
		Thus, we get:
		\begin{align*}
			\kappa([\delta,\delta'])
			&=-\projectionB \circ\delta\circ\delta'|_{\Gamma(A)}
			+\projectionB \circ\delta'\circ\delta|_{\Gamma(A)}\\
			&=\projectionB \circ\delta\circ\Big(\kappa(\delta')-\projectionA\circ \delta'|_{\Gamma(A)}\Big)\\
			&\qquad
			-\projectionB \circ\delta'\circ\Big(\kappa(\delta)-\projectionA\circ \delta|_{\Gamma(A)}\Big)\\
			&=\delta\triangleright\big(\kappa(\delta')\big)-\delta'\triangleright\big(\kappa(\delta)\big).
		\end{align*}	
		
		To show Equation \eqref{Eqt:temp2}, we begin with the case where $ {X}$ is merely a generating element: ${X}=b\in\Gamma(B)$.
		By Equation \eqref{Eqt:dABot} of $\dBott$ and the definition of $\mu_1$,
		we can examine the following identities for any $a\in \Gamma(A)$.
		\begin{align*}
			\big(\dBott (\delta\triangleright{b})\big)(a)
			=&\nabla_a(\delta\triangleright b )
			=\projectionB[a,\delta\triangleright b]_L
			=\projectionB[a,\projectionB\delta(b)]_L\\
			=&\projectionB[a,\delta(b)]_L
			=\projectionB\big(\delta[a,b]_L\big)-\projectionB[\delta(a),b]_L;\\
			\big(\delta\triangleright(\dBott {b})\big)(a)
			=&\delta\triangleright\big((\dBott b)(a)\big)-(\dBott b)(\delta\triangleright a)\\
			=&\delta\triangleright(\nabla_a b)-\nabla_{\delta\triangleright a}b\\
			=&\projectionB\big(\delta(\projectionB[a,b]_L)\big)-\projectionB\big[\projectionA\big(\delta(a)\big),b\big]_L.
		\end{align*}
		By the expression of $[\cdot,\cdot]_2$ in Proposition \ref{Prop:2and3bracket}, we have
\begin{align*}
			[\kappa(\delta),b]_2(a)&=\kappa(\delta)(\eth_b a)+[\kappa(\delta)(a),b]_B\\
			&=\projectionB\big(\delta(\projectionA[a,b]_L)\big)
			-\projectionB\big[\projectionB\big(\delta(a)\big),b\big]_L.
\end{align*}
		Hence, Equation \eqref{Eqt:temp2} holds for $X=b\in\Gamma(B)$. 
		The  verification of Equation \eqref{Eqt:temp2} for general $X\in\Omega^\bullet_A(B)$ follows from the property described in $(ii)$.  
		
		\item When $n=2$, we need to verify the following two equalities for 
		$\delta,\delta'\in\Derivations(L)$ and ${X},{X}_1,{X}_2\in\Omega_A^\bullet(B)$:
		\begin{equation*}
			[\delta,\delta']\triangleright{X}=\delta\triangleright\big(\kappa(\delta'),{X}\big)+\delta\triangleright(\delta'\triangleright{X})-\delta'\triangleright\big(\kappa(\delta),{X}\big)-\delta'\triangleright(\delta\triangleright{X}),
		\end{equation*}
	and
	\begin{align*}
				&\delta\triangleright(\dBott{X}_1,{X}_2)+(-1)^{1+\degree{{X}_1}\degree{{X}_2}}\delta\triangleright(\dBott{X}_2,{X}_1)+\delta\triangleright([{X}_1,{X}_2]_2)\\
				&\qquad=-[\kappa(\delta),{X}_1,{X}_2]_3
				+[\delta\triangleright{X}_1,{X}_2]_2\\
				&\qquad\qquad+(-1)^{1+\degree{{X}_1}\degree{{X}_2}}[\delta\triangleright{X}_2,{X}_1]_2
				-\dBott\big(\delta\triangleright({X}_1,{X}_2)\big).
		\end{align*}
		
		By the properties in $(iii)$, it suffices to consider the situations 
		where  ${X}=b\in\Gamma(B)$, ${X}_1=b_1\in\Gamma(B)$, and 
		${X}_2=b_2\in\Gamma(B)$. Using the definitions of $\kappa$ and $\mu_1$,
		we get
		\begin{align*}
			&\delta\triangleright\big(\kappa(\delta'),b\big)
			+\delta\triangleright(\delta'\triangleright b)
			-\delta'\triangleright\big(\kappa(\delta),b\big)
			-\delta'\triangleright(\delta\triangleright b)\\
			=&\kappa(\delta')\big(\projectionA\delta (b)\big)
			+\projectionB\delta\big(\projectionB\delta'(b)\big)
			-\kappa(\delta)\big(\projectionA\delta' (b)\big)
			-\projectionB\delta'\big(\projectionB\delta(b)\big)\\
			=&-\projectionB\delta'\big(\projectionA\delta (b)\big)
			+\projectionB\delta\big(\projectionB\delta'(b)\big)
			+\projectionB\delta\big(\projectionA\delta' (b)\big)
			-\projectionB\delta'\big(\projectionB\delta(b)\big)\\
			=&(\projectionB\circ\delta\circ\delta')(b)
			-(\projectionB\circ\delta'\circ\delta)(b)
			=[\delta,\delta']\triangleright b.
		\end{align*}
		Using the definitions of $\kappa$, $\mu_1$, $\mu_2$ and the expression of 
		$[\cdot,\cdot,\cdot]_3$ in Proposition \ref{Prop:2and3bracket},
		we get
		\begin{align*}
			&\delta\triangleright(\dBott b_1,b_2)-\delta\triangleright(\dBott b_2,b_1)+\delta\triangleright([b_1,b_2]_2)\\
			=&(\dBott b_1)\big(\projectionA\delta(b_2)\big)
			-(\dBott b_2)\big(\projectionA\delta(b_1)\big)
			+\projectionB\delta[b_1,b_2]_B\\
			=&\projectionB\big[\projectionA\delta(b_2),b_1\big]_L
			-\projectionB\big[\projectionA\delta(b_1),b_2\big]_L
			+\projectionB\delta\big(\projectionB[b_1,b_2]_L\big)\\
			=&-\projectionB[\projectionB \delta(b_2),b_1]_L
			+\projectionB[\projectionB \delta(b_1),b_2]_L
			-\projectionB\delta\big(\projectionA[b_1,b_2]_L\big)\\
			=&-[\delta\triangleright b_2,b_1]_2
			+[\delta\triangleright b_1,b_2]_2
			-[\kappa(\delta),b_1,b_2]_3
			-\dBott\big(\delta\triangleright(b_1,b_2)\big).
		\end{align*}
		
		\item When $n=3$, by the vanishing of $\mu_3$ and $[\cdots]_4$,
		we need to verify the following two equalities for
		$\delta,\delta'\in\Derivations(L)$ and ${X}_1,{X}_2,{X}_3\in\Omega_A^\bullet(B)$:
		\begin{align*}
			&\sum_{\sigma\in\Sh(2,1)}\chi(\sigma;{X}_1,{X}_2,{X}_3) 
			\delta\triangleright
			\big([{X}_{\sigma(1)},{X}_{\sigma(2)}]_2,{X}_{\sigma(3)}\big)
			+\delta\triangleright[{X}_1,{X}_2,{X}_3]_3\\
			=&\sum_{\sigma\in\Sh(1,2)}\chi(\sigma;{X}_1,{X}_2,{X}_3)
			\big[\delta\triangleright {X}_{\sigma(1)},{X}_{\sigma(2)},{X}_{\sigma(3)}\big]_3\\
			&-\sum_{\sigma\in\Sh(2,1)}\chi(\sigma;{X}_1,{X}_2,{X}_3)
			\big[\delta\triangleright({X}_{\sigma(1)},{X}_{\sigma(2)}),{X}_{\sigma(3)}\big]_2\,,
		\end{align*}
	and
	\begin{align*}
			&\sum_{\sigma\in\Sh(2,1)}\chi(\sigma;{X}_1,{X}_2,{X}_3)\Big(\delta\triangleright\big(\delta'\triangleright({X}_{\sigma(1)},{X}_{\sigma(2)}),{X}_{\sigma(3)}\big)\Big)\\
			&\qquad=\sum_{\sigma\in\Sh(2,1)}\chi(\sigma;{X}_1,{X}_2,{X}_3)\Big(\delta'\triangleright\big(\delta\triangleright({X}_{\sigma(1)},{X}_{\sigma(2)}),{X}_{\sigma(3)}\big)\Big).
		\end{align*}
		In fact, when the three $X_i$'s are of the form $b_i\in\Gamma(B)$,
		all the terms in these equalities are trivial.
		For general $X_i$, one resorts to (5) and (6) of Theorem 
		\ref{Thm:Linftystructuregenerator}, and properties  $(ii)$ and $(iii)$.
		
		\item When $n\geqslant4$, by the vanishing of $\mu_{k}$ and $[\cdots]_{k+1}$ for $k\geqslant 3$, we are only left to verify the following equality for $\delta\in\Derivations(L)$ and ${X}_1,\cdots,{X}_4\in\Omega_A^\bullet(B)$:
		\begin{align*}
			&\sum_{\sigma\in\Sh(3,1)}\chi(\sigma;{X}_1,\cdots,{X}_4)
			\delta\triangleright\big([{X}_{\sigma(1)},{X}_{\sigma(2)},{X}_{\sigma(3)}]_3,{X}_{\sigma(4)}\big)\\
			=&-\sum_{\sigma\in\Sh(2,2)}\chi(\sigma;{X}_1,\cdots,{X}_4)
			\big[\delta\triangleright({X}_{\sigma(1)},{X}_{\sigma(2)}),{X}_{\sigma(3)},{X}_{\sigma(4)}\big]_3.
		\end{align*}
	\end{itemize}		The argument is similar to the $n=3$ case.
\end{proof}

Following Theorems \ref{Thm:hplusg} and \ref{MainTheorem}, we have a corollary.
\begin{Cor}\label{Cor:main}
	The space $\Derivations(L)\oplus\Omega^\bullet_A(B)$ admits an ${\Linftythree}$ algebra structure which extends the ${\Linftythree}$ structure on $\Omega^\bullet_A(B)$.
\end{Cor}

Define a subspace of $\Derivations(L)$:
\[\Derivations(L,A):=\ker\kappa=\{\delta\in \Derivations(L)~|~ \delta \bsection{A}\subset \bsection{A}\}.\]
Following Proposition \ref{Prop:hactionpassingtoH}, we have another corollary. 
\begin{Cor}\label{Cor:actiononH} The above $\Derivations(L,A)$ is a Lie algebra and it acts by derivation on the graded Lie algebra $H\big(\Omega^\bullet _A(B),[\cdot]_1=\dBott\big)=H_{\mathrm{CE}}(A;B)$.
\end{Cor}

\begin{Ex}
	Let $\mathfrak{h}$ be a complex semisimple Lie algebra,
	and $\mathfrak{t}$ a Cartan subalgebra in $\mathfrak{h}$.
	Let $\mathfrak{h}=\mathfrak{t}\oplus \bigoplus_{\alpha\in\Delta}\mathfrak{h}_\alpha$ be its root decomposition, 
	where $\Delta\subset \mathfrak{t}^\vee$ is the root system of $\mathfrak{h}$.
	Fix a set $\Pi$ of simple roots and denote the set of positive roots by $\Delta^+$.
	For all $\alpha\in\Delta^+$, there exist $e_\alpha\in \mathfrak{t}$, $x_\alpha\in \mathfrak{h}_\alpha$, and $x_{-\alpha}\in \mathfrak{h}_{-\alpha}$
	such that
	$$[e_\alpha,x_\alpha]=2x_\alpha,\quad [e_\alpha,x_{-\alpha}]=-2x_{-\alpha},\quad [x_{\alpha},x_{-\alpha}]=e_{\alpha}.$$
	%Suppose that $H^\vee$ is spanned by $e_{\alpha}^\vee$ in duality to $e_{\alpha}$ for those simple roots $\alpha\in \Delta^0$. 
	Furthermore, there are two standard sets of structure constants $C_{\alpha,\beta},N_{\alpha,\beta}\in \mathbb{Z}$ such that (see \cite[Chapters III,VII]{Humphreysbook}):
	\begin{gather*}
		[e_{\alpha},x_{\beta}]=C_{\alpha,\beta} x_{ \beta},
		\quad
		[e_{\alpha},x_{-\beta}]=-C_{\alpha,\beta} x_{ -\beta},
		\quad \forall \alpha,\beta\in \Delta;\\
		[x_{\alpha},x_{\beta}]=N_{\alpha,\beta} x_{\alpha+\beta},
		\quad
		[x_{-\alpha},x_{-\beta}]=-N_{\alpha,\beta} x_{-\alpha-\beta},
		\quad 
		\forall \alpha,\beta\in \Delta \mbox{ with } \alpha+\beta \in \Delta.
	\end{gather*}
	For the simple roots $\alpha,\beta\in\Pi$ under a given ordering,
	the data $C_{\alpha,\beta}$ form  the corresponding Cartan matrix.
	For $\alpha,\beta\in \Delta \mbox{ with } \alpha+\beta \notin \Delta\cup\{0\}$, set $N_{\alpha,\beta}=0$. However, since $[x_{\alpha},x_{-\alpha}]=e_{\alpha} \notin \bigoplus_{\beta\in\Delta}\mathfrak{h}_\beta$ for $\alpha\in \Delta$, the number $N_{\alpha,-\alpha}$ is undefined.
	
	Now, the Lie pair we take is $L=\mathfrak{h}$
	together with its subalgebra $A=\mathfrak{t}$,
	and hence $B=L/A=\bigoplus_{\alpha\in\Delta}\mathfrak{h}_\alpha$.
	According to Theorem \ref{Thm:Linftystructuregenerator} and Proposition \ref{Prop:2and3bracket}, the space $\Lambda^\bullet \mathfrak{t}^\vee\otimes (\bigoplus_{\alpha\in\Delta}\mathfrak{h}_\alpha)$ 
	admits an ${\Linftythree}$ structure, see also \cite[Theorem 4.22]{BCSX}.  
	
	Since $\mathfrak{h}$ is semisimple,
	we have $\Derivations(\mathfrak{h})\cong \mathfrak{h}$. 
	Therefore, by Theorem \ref{MainTheorem},
	there is a Lie algebra action of $\mathfrak{h}$ on
	$\Lambda^\bullet \mathfrak{t}^\vee\otimes (\bigoplus_{\alpha\in\Delta}\mathfrak{h}_\alpha)$.
	The action maps are specified as follows:
	\begin{itemize}
		\item[(1)] The $0$-action $\mu_0=\kappa:~ \mathfrak{h}\to \mathfrak{t}^\vee\otimes (\bigoplus_{\alpha\in\Delta}\mathfrak{h}_\alpha)$  is  given by 
		$$\kappa(e_{\alpha})=0,
		\quad  \kappa(x_{\alpha})=\sum_{\beta\in \Pi} C_{\beta,\alpha} e^{\vee}_{\beta}\otimes x_{\alpha}, 
		\quad \kappa(x_{-\alpha})=-\sum_{\beta\in \Pi} C_{\beta,\alpha} e^{\vee}_{\beta}\otimes x_{-\alpha},$$
		$  \forall\alpha\in \Delta^+$.
		\item[(2)] The $1$-action  $
		\mu_1:~ \mathfrak{h}\times  \big(\Lambda^\bullet \mathfrak{t}^\vee\otimes (\bigoplus_{\alpha\in\Delta}\mathfrak{h}_\alpha)\big) \to  \Lambda^\bullet \mathfrak{t}^\vee\otimes (\bigoplus_{\alpha\in\Delta}\mathfrak{h}_\alpha)  $ % ($k\geqslant  0$)
		~is given by the generating relations
		\[
		\begin{cases}
			x_{\alpha}\triangleright x_{-\alpha} = 0 & \forall \alpha\in \Delta,\\
			x_{\alpha}\triangleright x_{\beta}=N_{\alpha,\beta}x_{\alpha+\beta}
			&   \forall \alpha,\beta\in \Delta
			\mbox{ with } \alpha+\beta\neq 0,\\
			e_{\alpha}\triangleright x_{\beta}=C_{\alpha,\beta}x_{\beta}
			& \forall \alpha\in \Pi,\beta\in \Delta,
		\end{cases}
		\]
		with vanishing $\varrho_1=0:~\mathfrak{h}\to{\rm Der}^0(\Lambda^{\bullet} \mathfrak{t}^\vee)$. 
		\item[(3)] The $2$-action $
		\mu_2:~ \mathfrak{h}\times  \big(\Lambda^\bullet \mathfrak{t}^\vee\otimes (\bigoplus_{\alpha\in\Delta}\mathfrak{h}_\alpha)\big)^{\otimes 2} \to  \Lambda^{\bullet-1} \mathfrak{t}^\vee\otimes (\bigoplus_{\alpha\in\Delta}\mathfrak{h}_\alpha)  $
		is given by the generating relation $$\mathfrak{h}\triangleright\Big(\bigoplus_{\alpha\in\Delta}\mathfrak{h}_\alpha,\bigoplus_{\alpha\in\Delta}\mathfrak{h}_\alpha\Big)=0,$$
		with
		$\varrho_2:~\mathfrak{h}\times \big(\Lambda^\bullet \mathfrak{t}^\vee\otimes (\bigoplus_{\alpha\in\Delta}\mathfrak{h}_\alpha)\big)\to {\rm Der}^{\bullet-1}(\Lambda^\bullet \mathfrak{t}^\vee)$
		$$
		\begin{cases}
			\varrho_2(x_{\alpha},\omega \otimes x_{-\alpha})=(-1)^{|\omega|+1}\omega\cdot e_\alpha \lrcorner \quad
			& \forall \alpha \in \Delta,\\
			\varrho_2(x_{\alpha},\omega \otimes x_{\beta})=0
			& \forall \alpha,\beta\in \Delta
			\mbox{ with } \alpha+\beta\neq 0,\\
			\varrho_2(e_{\alpha},\omega \otimes x_{\beta})=0 
			& \forall \alpha,\beta\in \Delta,
		\end{cases}
		$$
		where $\omega\in\Lambda^\bullet \mathfrak{t}^\vee$.
		\item[(4)] All higher  actions $\mu_n$ ($n\geqslant 3$) are trivial.
	\end{itemize}
	
\end{Ex}

\begin{Rem}
	In Remark \ref{Rmk:depedenceofsplittings} we have addressed the dependence of splittings  $L=A\oplus B$ which give different but isomorphic $\Linftythree$ algebras $\Omega^\bullet_A(B)$. We point out that the $\Derivations(L)$-action on $\Omega^\bullet_A(B)$ constructed by Theorem \ref{MainTheorem} also depends on the choice of splitting. However, one can prove that such  $\Derivations(L)$-actions are \emph{compatible} with the $\Linftythree$ algebras $\Omega^\bullet_A(B)$ arising from different splittings (see Definition \ref{Def: hactioncompatible}). For this reason,  the $\Derivations(L)$-action given by the theorem is in fact \emph{canonical}.  
\end{Rem}

\section{Gauge equivalences of Maurer-Cartan elements based on Lie algebra actions}
In this part, we review a notion of gauge equivalence of Maurer-Cartan elements introduced by Getzler \cite{Getzler}, and will propose another type of gauge equivalence following  Getzler's formula.

\textbf{Notation:}  Let $\v$ be a local Artinian  $\K$-algebras with residue field ${\K}$. We will denote by $\maximalidealofartin $ the maximal ideal of $\v$. 

Let  $\g$ be an  $L_\infty$ algebra. The graded vector space  $\g\otimes\maximalidealofartin $ now becomes a nilpotent\footnote{This means that the lower central series     $F^i (\g\otimes\maximalidealofartin) $   vanish for  $i $ sufficiently large, where $F^1 (\g\otimes\maximalidealofartin)=\g\otimes\maximalidealofartin$ and, for $i\geqslant 2$,
	$$F^i (\g\otimes\maximalidealofartin):=\sum \limits_{i_1+\cdots+i_k=i}[F^{i_1}(\g\otimes\maximalidealofartin),\cdots,F^{i_k}(\g\otimes\maximalidealofartin)]_k\,.$$} $L_\infty$ algebra whose structure maps are extended from $\g$.

\begin{Def}\label{Def:MC}
	A  \textbf{Maurer-Cartan  element} in an  $L_\infty$ algebra  $\g$ (with coefficient $\maximalidealofartin $)  is an element $\xi\in \g^1\otimes \maximalidealofartin  $ such that
	\begin{equation*}
		\sum\limits_{k=1}^\infty \dfrac{1}{k!}[\xi,\cdots,\xi]_k=0.
	\end{equation*}
	We will denote by ${\rm MC}_{\v}(\g)$ $(\subset \g^1\otimes \maximalidealofartin  )$
	the set of Maurer-Cartan elements in $\g$.
\end{Def}
For any element $\xi \in \g^1\otimes \maximalidealofartin $, the formula
\begin{align*}
	[g_1,\cdots,g_i]^\xi_i
	&=\sum\limits_{k=0}^\infty \frac{1}{k!}[\xi^{\wedge k},g_1,\cdots,g_i]_{i+k}\\
	&=[g_1,\cdots,g_i]_i+[\xi,g_1,\cdots,g_i]_{i+1}+
	\frac{1}{2}[\xi,\xi,g_1,\cdots,g_i]_{i+2}+\cdots
\end{align*}
defines a new sequence of brackets on $\g\otimes \maximalidealofartin  $ known as the ($i$-th) $\xi$-bracket, where
$[\xi^{\wedge k},g_1,\cdots,g_i]_{k+i}$ is an abbreviation for
$[\xi,\cdots,\xi,g_1,\cdots,g_i]_{k+i}$, in which $\xi$ occurs $k$ times.
For example, we have
$$
[g]^\xi_1
=[g]_1+[\xi,g]_2+
\frac{1}{2}[\xi,\xi,g]_3+\cdots.
$$

For $b \in \g^0\otimes \maximalidealofartin  $, $\xi\in {\rm MC}_{\v}(\g)$, define  $e^b *\xi\in \g^1\otimes \maximalidealofartin $ by
$$e^b *\xi:=\xi-\sum \limits_{k=1}^\infty \frac{1}{k!} e^k_\xi(b),$$
where $ e^k_\xi(b)$ are inductively determined by
$$e^1_\xi(b)=[b]^\xi_1, $$
$$e^{k+1}_\xi(b)=\sum\limits_{n=1}^k \frac{1}{n!}
\sum\limits_{\begin{subarray}{c}
		k_1+\cdots+k_n=k \\ k_i\geqslant 1
\end{subarray} }\frac{k!}{k_1!\cdots k_n!}
[b,e^{k_1}_\xi(b),\cdots,e^{k_n}_\xi(b)]^\xi_{n+1}.$$	
It is shown \textit{op. cit.} that $e^b *\xi$ again lands in ${\rm MC}_{\v}(\g)$ and hence one obtains an ``action'' of $\g^0\otimes \maximalidealofartin $ on ${\rm MC}_{\v}(\g)$. Let us call it the \textit{gauge action}. Note that $\g^0\otimes \maximalidealofartin $ is \emph{not} a Lie algebra in general.

\begin{Def}\label{gauge equivalent}Let $\g$ be an $\Linfty$ algebra.
	Two Maurer-Cartan elements $\xi, \eta \in {\rm MC}_{\v}(\g) $ are said to be \textbf{gauge equivalent} if there exists an element $b \in  \g^0\otimes \maximalidealofartin   $
	such that
	$   e^b *\xi=\eta $.
\end{Def}

Following the recipe of gauge actions, we propose another type of gauge equivalence of Maurer-Cartan elements arising from Lie algebra actions. Suppose that the $\Linfty$ algebra $\g$ admits an action by a Lie algebra $\h$, with the structure maps being $\{\mu_n\}$ (or $\triangleright$). For any $h\in \h\otimes \maximalidealofartin  $ and $\xi\in {\rm MC}_{\v}(\g) $, define $e^{h}* \xi\in \g^1\otimes \maximalidealofartin $ as follows:
$$e^{h}* \xi:=  \xi-\sum \limits_{k=1}^\infty \frac{1}{k!} e^k_\xi(h ),$$
where $e^k_\xi(h )\in \g^1\otimes \maximalidealofartin $ ($k\geqslant 1$) is inductively defined by
\begin{gather*}
	e^1_\xi(h )
	=\sum_{i=0}^\infty\frac{1}{i!}\mu_{i}(\xi^{\wedge i},h)
	=\kappa(h)-h\triangleright(\xi) 
	+\frac{1}{2}h\triangleright(\xi,\xi)
	-\frac{1}{6} h\triangleright(\xi,\xi,\xi)
	+\cdots, \label{Eqt:e1xib0}
\end{gather*}
and
\begin{gather*}
	e^{k+1}_\xi(h )
	=\sum\limits_{n=1}^k \frac{1}{n!}
	\sum\limits_{\begin{subarray}{c}
			k_1+\cdots+k_n=k \\ k_i\geqslant 1
	\end{subarray} }\frac{k!}{k_1!\cdots k_n!}
	\sum\limits_{j=0}^\infty \frac{(-1)^j}{j!}
	h \triangleright \big(\xi^{\wedge j} ,e^{k_1}_\xi(h),\cdots,e^{k_n}_\xi(h)\big) \,. \label{Eqt:enxib0}
\end{gather*} 

Again, one can prove that $e^{h}* \xi$ belongs to ${\rm MC}_{\v}(\g) $. In turn, we  obtain a new type of gauge action which is given by the Lie algebra $\h\otimes \maximalidealofartin $ on ${\rm MC}_{\v}(\g) $, and hence the notion of $\h$-gauge equivalence:
\begin{Def}\label{Def:WeakGaugeEquivalence}
	With the assumptions as above, two Maurer-Cartan elements $\xi$ and $\eta\in {\rm MC}_{\v}(\g)$ are said to be \textbf{$\h$-gauge equivalent} if there exists an $h \in \h\otimes \maximalidealofartin $ such that
	$e^{h}* \xi=\eta$.
\end{Def}
More properties of   $\h$-gauge equivalences will be studied in the future. Just for this note, let us turn back to the settings of Section \ref{Sec:Mainresults},  Theorem \ref{MainTheorem} in particular:  $\g$ is the $\Linftythree$ algebra $\Omega^\bullet_A(B)$ with an $\h=\Derivations(L)$ action. We have $\g^0=\Gamma(B)$. For an element $b\in \Gamma(B)\otimes \maximalidealofartin $, consider the operator
$$\adLb:=[b,\cdot]_L \in \mathrm{End}_{  \v }( \Gamma(L)\otimes \maximalidealofartin  ) .$$  
We can check that, $\adLb$ is indeed a derivation, i.e. $ \adLb\in \Derivations(L)\otimes \maximalidealofartin$, and moreover, we have
\begin{eqnarray*}
	\kappa(\adLb) &=& [b]_1  (= \dBott b ),\\
	{ \adLb \triangleright ( X)   } &=& {[b,X]_2},
	\\
	{ \adLb \triangleright ( X,Y)  } &=& {[b,X,Y]_3,}
\end{eqnarray*} for all $X$ and $Y\in \Omega_A^\bullet(B)  \otimes \maximalidealofartin $. From these relations, we can  prove the following fact.

\begin{Thm}\label{Thm:adb=baction} For any $b\in \Gamma(B)\otimes \maximalidealofartin$,       the  $\Derivations(L)\otimes \maximalidealofartin$-gauge action by $\delta =\adLb  $ coincides with the   gauge action by $b$, i.e.
	$$e^{\adLb}*\xi=e^b*\xi,\quad \forall \xi \in {\rm MC}_{\v}\big(\Omega_A^\bullet(B)\big).$$
\end{Thm}

From this theorem, we  see that our newly introduced $\Derivations(L)\otimes \maximalidealofartin$-gauge action recovers  the  classical one  given by Getzler in \cite{Getzler}. It comes from the $\Derivations(L)$-action on $\Omega_A^\bullet(B)$ as detailed in the previous Theorem \ref{MainTheorem}. We have a reason to call such actions $\delta\triangleright(\cdots)$ ($\delta\in \Derivations(L)$) internal symmetries --- the Lie algebra $\Derivations(L)$ exists naturally; while the space $\Gamma(B)$ is not a Lie algebra in general, and  its role   depends on the splitting $L\cong A\oplus B$. Of course,  isomorphism classes of ${\rm MC}_{\v}\big(\Omega_A^\bullet(B)\big)$ up to $\Derivations(L)\otimes \maximalidealofartin$-gauge equivalences (i.e. the associated deformation space)  will have fewer elements than up to $\Gamma(B)\otimes \maximalidealofartin$-gauge equivalences. The new problems caused by this way of thinking await our follow-up research.

\end{document}